\documentclass[12pt,reqno]{amsart}
\usepackage{amsmath,amsfonts,amssymb,amscd,amsthm,amsbsy,epsf}
\usepackage[dvips]{graphicx}
\usepackage[usenames]{xcolor}
\usepackage{comment}
\usepackage{tikz}
\usepackage{ulem}
\newcommand{\sgn}{\operatorname{sgn}}

\usepackage{subcaption}
\footskip=18pt \numberwithin{equation}{section}
\newcounter{Propositionscounter}
\newcounter{Definitionscounter}

\newtheorem{theorem}{Theorem}
\newtheorem{meta-thm}[theorem]{Meta-Theorem}

\newtheorem{proposition}[Propositionscounter]{Proposition}

\newtheorem{definition}[Definitionscounter]{Definition}

\DeclareMathAlphabet{\mathcalligra}{T1}{calligra}{m}{n}

\newcommand\beq[1]{ \begin{equation}\label{#1} }
	\newcommand{\eeq}{ \end{equation} }
\newcommand{\beqno}{ \[ }
\newcommand{\eeqno}{ \] }
\newcommand\beqa[1]{ \begin{eqnarray} \label{#1}}
	
	\newcommand{\eeqa}{ \end{eqnarray} }

\newcommand{\beqano}{ \begin{eqnarray*} }
	
	\newcommand{\eeqano}{ \end{eqnarray*} }

\newcommand\equ[1]{{\rm (\ref{#1})}}

\newcommand{\norm}[1]{\left\lVert#1\right\rVert}
\newcommand{\abs}[1]{\left|#1\right|}

\def\u{\underline}

\def\A{{\mathcal A}}

\def\E{{\mathcal E}}

\def\H{{\mathcal H}}
\def\N{{\mathcal N}}
\def\K{{\mathcal K}}

\def\R{{\mathcal R}}
\def\S{{\mathcal S}}

\def\Z{{\mathcal Z}}
\def\L{{\mathcal L}}
\def\complex{{\mathbb C}}
\def\integer{{\mathbb Z}}

\def\real{{\mathbb R}}
\def\torus{{\mathbb T}}
\def\zed{{\mathbb Z}}

\def\eps{\varepsilon}

\DeclareFontFamily{U}{matha}{\hyphenchar\font45}
\DeclareFontShape{U}{matha}{m}{n}{
      <5> <6> <7> <8> <9> <10> gen * matha
      <10.95> matha10 <12> <14.4> <17.28> <20.74> <24.88> matha12
      }{}
\DeclareSymbolFont{matha}{U}{matha}{m}{n}

\DeclareMathSymbol{\varleftarrow}{3}{matha}{"D0}
\DeclareMathSymbol{\varrightarrow}{3}{matha}{"D1}

\begin{document}

\title[Effective stability estimates close to resonances]
{Effective stability estimates close to resonances with applications to rotational dynamics}
	
\author[A. Celletti]{Alessandra Celletti}
\address{Department of Mathematics, University of Roma Tor Vergata, Via della Ricerca Scientifica 1, 00133 Roma (Italy)}
\email{celletti@mat.uniroma2.it}
\author[A. Dogkas]{Anargyros Dogkas}
\address{
		Department of Mathematics, University of Pisa, Largo Bruno Pontecorvo, 5, 56127 Pisa,(Italy)}
	\email{anargyros.dogkas@phd.unipi.it}
	\author[A. Francesca Guido]{Alessia Francesca Guido$^{**}$}
    \address{Department of Physics, University of Trento, Via Sommarive
14, 38123 Povo, Trento (Italy)}
	\address{Department of Mathematics, University of Roma Tor Vergata, Via
		della Ricerca Scientifica 1, 00133 Roma (Italy)}
	\email{alessia.guido@unitn.it}
    \thanks{$^{**}$ Corresponding author Alessia Francesca Guido (alessia.guido@unitn.it)}
	
	\baselineskip=18pt              
	
\begin{abstract}

We consider nearly-integrable Hamiltonian systems defined over a non-resonant domain. In the neighborhood of resonances, we use Nekhoroshev-like estimates to provide effective stability bounds for the action variables over long time. The applicability conditions of these estimates allow some freedom in the choice of parameters. Hence, we develop an optimization algorithm for choosing parameters that maximize the stability time. To further improve the stability estimates, we use perturbation theory to reduce the norm of the perturbing function. We implement this procedure (effective stability estimates and perturbation theory) to analyze the stability of sequences of irrational (Diophantine) frequencies converging to frequencies corresponding to resonances. We consider two applications to models describing problems of rotational dynamics in Celestial Mechanics: the spin-orbit problem, described by a 1D time-dependent Hamiltonian, and the spin-spin-orbit model, described by a 2D time-dependent Hamiltonian. We show stability results for orbits close to the main resonances associated with such models.

\end{abstract}

\maketitle	
		
	
\section{Introduction}\label{sec:introduction}

Invariant tori and periodic orbits are the backbone of the dynamics of nearly-integrable Hamiltonian systems. These systems can be conveniently described in action-angle coordinates as the sum of an integrable part (depending just on the actions) and a perturbation which is assumed to be small (in norm) with respect to the integrable part.

Invariant tori, formed by quasi-periodic motions, can be established through the well-known Kolmogorov-Arnold-Moser (hereafter, KAM) theory (\cite{Kolmogorov54}, \cite{Arnold63a}, \cite{Moser62}), which, in particular, provides the persistence under perturbation of an invariant torus with fixed frequency vector. Nekhoroshev's theorem (\cite{Nekhoroshev77}) gives different results than KAM theory and allows one to prove that, under certain conditions, all motions of a real-analytic, steep Hamiltonian system are stable (i.e., the action variables remain close to their initial conditions) for exponentially long time. Several works focus on improving the exponential stability exponents, which determine the order of magnitude for both the stability time and the bounds on the actions. In the general case, an optimization of the stability exponents is shown in \cite{GuChieBen2014}, where it is conjectured that these improvements are optimal. However, for the special cases, that satisfy the stronger $\ell$, $m$ quasi-convexity condition a heuristic analysis of how to improve the stability exponents can be found in \cite{Chirikov79}, while improvements using different analytical approaches are given in \cite{Lochack92}, \cite{Lochak},  \cite{Poschel93} (see also \cite{Bounemoura2011}). Such improvements are nearly optimal for the convex and quasi-convex cases, as it results from the analysis of Arnold diffusion (compare with \cite{Zhang2017}, \cite{Zhang2011}).

The steepness condition, or equivalently the quasi-convexity condition, required by Nekhoroshev's theorem, is important for bounding the evolution of an initial condition that is in the proximity of a resonance. Resonant periodic orbits occur when the frequencies, associated with the integrable part of the Hamiltonian, satisfy a commensurability relation with integer coefficients. The strict relation between invariant tori and periodic orbits is readily explained when the frequency of the invariant torus is a scalar. In this case, the periodic orbits approximating the invariant torus have frequencies given by the rational approximants to the irrational frequency of the torus (see, e.g., \cite{Greene79}).

In contrast, non-resonant stability estimates can provide a confinement of the motion without the need for a steepness condition. As a consequence, better results can be obtained in the vicinity of a resonance as long as the initial condition is sufficiently far from the commensurability. Therefore here, we take a different approach and we rather approximate periodic orbits by sequences of invariant tori, aiming to study the stability in the vicinity of resonances. In particular, for a given resonant frequency, we compute approximating irrationals that satisfy the Diophantine condition, which guarantees that we approach the resonance without intersecting any other commensurabilities. For those non-resonant frequencies, we compute effective stability estimates to provide a bound for the actions over long times, using the analytical estimates given in \cite{Poschel93}. However, such stability estimates depend on multiple parameters, that must satisfy certain applicability conditions. Since these conditions allow some freedom in the selection of the parameters, we propose an optimization algorithm to choose the parameters that maximize the stability time.

As it is common in KAM theory and Nekhoroshev’s theorem, before implementing the stability estimates, it is usually convenient to apply perturbation theory to reduce the size of the perturbation. This procedure allows us to get a system closer to an integrable one (see, for example, \cite{Celletti90I}, while we refer to \cite{CDE}, \cite{CF}, \cite{CG}, \cite{GS}, for application of Nekhoroshev’s theorem to models of Celestial Mechanics).

Effective stability estimates, together with the implementation of perturbation theory, are applied to models of rotational dynamics in Celestial Mechanics. The first model is called the spin-orbit problem (\cite{Celletti90I}, \cite{goldreich1966spin}, \cite{Lei2024}, \cite{Wisdom}, see also \cite{PSV2024} for a model including dissipation), which is represented by a 1D time-dependent Hamiltonian function and describes the motion of a satellite with triaxial shape moving on a Keplerian orbit around a central planet while rotating around its perpendicular axis to the orbital plane. The second problem is called the spin-spin-orbit problem (\cite{boue},  \cite{BCL25}, \cite{Maciejewski}, \cite{Misquero}), which is described by a 2D time-dependent Hamiltonian function and describes the motion of two ellipsoids moving around each other on Keplerian orbits and rotating about the respective spin axes, assumed to be perpendicular to the orbital plane. In both sample models, we provide stability estimates around the main resonances, namely the resonances associated with the original perturbation (while we will refer to secondary resonances as those stemming after the implementation of perturbation theory). Our results show a complex interaction between main resonance, secondary resonances, and invariant tori,  as well as the efficiency of our method and the applicability to real astronomical systems. 

This work is organized as follows. In Section~\ref{sec:Theory}, we introduce 
nearly-integrable Hamiltonian systems, resonances, and the Diophantine condition. 
We present effective stability estimates in non-resonant domains following one of the results given in \cite{Poschel93}. This result depends on several parameters, which are chosen following an optimization algorithm detailed in Appendix~\ref{AppendixAlgorithm}. We also provide a brief description of perturbation theory, which is used to normalize the Hamiltonian and reduce the magnitude of the perturbation. Finally, we discuss the choice of sequences of irrational (Diophantine) frequencies that approach specific resonances. In Section~\ref{sec:results}, we present results around the most important resonances in the spin-orbit and the spin-spin-orbit models. A series expansion of the corresponding Hamiltonians, truncated to a finite order in the eccentricity (which enters the model as a parameter), {is} given in Appendix~\ref{app:spin}. Having fixed specific resonances, we select irrational frequencies converging to these resonances and we compute effective stability estimates (with bounds on the norms of the perturbing function detailed in Appendix~\ref{app:norm}), either on the original Hamiltonian or on the Hamiltonian obtained {by} implementing perturbation theory. 
Conclusions are drawn in Section~\ref{sec:conclusion}.

\section{Optimal effective stability estimates around resonances}\label{sec:Theory}

The action variables of a nearly-integrable Hamiltonian system, defined over a non-resonant domain (Section~\ref{sec:Resonances}), can be confined for an exponential amount of time (Section~\ref{sec:effective}) under some constraints on the smallness of the norm of the perturbing function. In physical systems, the magnitude of the perturbation can be too large to satisfy these constraints. In such cases, it is convenient to use perturbation theory (Section~\ref{sec:PertTheory}) to reduce the norm of the perturbing function through a canonical near-identity transformation. 

Our goal is in fact to provide stability estimates for the action variables that maximize the stability time, encompassing also the neighborhood of some resonances. To avoid  \sl secondary \rm commensurabilities, namely resonances that arise after the {implementation} of perturbation theory, we consider sequences of non-resonant Diophantine frequencies (Section~\ref{sec:toripo}) that asymptotically converge to a pre-assigned resonance.

\subsection{Nearly-integrable Hamiltonian systems and resonances}\label{sec:Resonances}

Let us consider a nearly-integrable Hamiltonian system, written in terms of some action-angle variables $(\u{p},\u{q})\in D\subseteq \real^n\times\torus^n$, with the projection $P$ of $D$ in the action space $P D\subseteq\real^n$ an open domain:
\beq{GeneralHamiltonian}
    \H(\u{p},\u{q})=h(\u{p}) + \R(\u{p},\u{q}),
\eeq
where $h(\u{p})$ is the integrable part and $\R(\u{p},\u{q})$ is the perturbation with $\norm{\R(\u{p},\u{q})}_D$ sufficiently smaller than $\norm{h(\u{p})}_D$, for some function norm (to be specified later) $\norm{\cdot}_D$ over the domain $D$. We call $\omega(\u{p})=\nabla_{\u{p}}h(\u{p})$ the frequency  associated to the integrable part $h(\u{p})$. 

We assume that the perturbing function is a trigonometric function with a finite number of  Fourier coefficients, truncated\footnote{{For the models we are going to analyze (spin-orbit and spin-spin-orbit problems), 
the perturbing parameter $\varepsilon$ (measuring the oblateness of the body) is less than unity and the coefficients of the perturbing function decay as powers of the eccentricity, which is assumed to be less than one. As a result, {higher-order} harmonics are of significantly smaller magnitude and they are negligible for the dynamics. For this reason, the perturbing function is typically truncated to some finite order.}} to some frequency $K$:
$$
\R(\underline{p},\underline{q})=\sum_{\underline{k}\in\zed_{K}^n}R_{\underline{k}}(\underline{p})\exp(i\underline{k}\cdot\underline{q})\ ,\qquad \zed_{K}^n=\left\{\underline{k}\in\zed^n : \norm{\underline{k}}_1\leq K\right\}
$$
with $\norm{\underline{k}}_1=|k_1|+...+|k_n|$. To make explicit the set of Fourier coefficients appearing in $\R$, we write 
$$
\R(\underline{p},\underline{q})=\R_{\underline{k}\in\zed_{K}^n}(\underline{p},\underline{q})\ . 
$$
Any non-autonomous system of dimension $n-1$ (like those that will be considered in Section~\ref{sec:results}) can be transformed into an autonomous system of dimension $n$ by extending the phase space through the variable $q_n=\nu t$ with $\nu$ some constant frequency and its {conjugate} momentum $p_n$; as a consequence, one needs to add the linear term $\nu p_n$ to the integrable part. 

In the definitions below we introduce the notions of resonance, non-resonant domain and Diophantine frequency.  

\begin{definition}\label{resgeneral}
    For a fixed value of the actions $\u{p}=\u{p}_0$, the frequency vector $\u{\omega}(\u{p}_0)$ satisfies a resonance of order $\u{k}$, with $\u{k}=(k_1,k_2,\dots,k_{n})\in\zed^n\backslash\{\u{0}\}$, if
    \beq{resonance_condition}
        \u{k}\cdot\u{\omega}(\u{p}_0)=0\ .
    \eeq
\end{definition}

We notice that in the non-autonomous one-dimensional case, in which the frequency is given by $\u \omega(p_1)=(\partial_{p_1} h(p_1),\nu)$, a resonance of order  $(k_1,k_2)$ is equivalent to the existence of a periodic orbit such that 
$$
q_1(t+2\pi k_1) = q_1(t) - 2\pi\nu k_2 \ .
$$

\begin{definition}\label{nonres}
The domain $D\subseteq \real^n\times\torus^n$ is called $\alpha$, $K$ non-resonant modulo $\Lambda$, for {$\Lambda\subseteq\zed_{K}^n\backslash\{\underline{0}\}$}, if 
    $$
    \abs{\underline{k}\cdot\underline{\omega}(\underline{p})}\geq\alpha\ \quad \forall\underline{k}\in \zed_K^n\backslash{(\Lambda\cup\{\underline{0}\})}\ , \quad (\underline{p},\underline{q})\in D
    $$
    for some $\alpha\in\real_+$.
    Furthermore, if $\Lambda=\emptyset$, then $D$ is said to be a completely $\alpha$, $K$ non-resonant domain. 
\end{definition}

The Diophantine condition introduced below gives a stronger non-resonant condition on the frequency. 

\begin{definition}\label{diophcond}
A frequency vector $\u{\omega}(\u{p})$ is said to satisfy the Diophantine condition if 
$$
\abs{\underline{k}\cdot\underline{\omega}(\underline{p})}\geq\frac{C}{\norm{k}^\tau_1}\quad \forall\ \u{k}\in\zed^{n}\backslash\{\u{0}\}
$$
for some constants $C\in\real_+$ and $\tau\geq n-1$.
\end{definition}

The main difference between the non-resonant definitions \ref{nonres} and \ref{diophcond} is that, while the $\alpha$, $K$ non-resonant condition requires that the frequency is away from the commensurabilities up to some given frequency $K$, the Diophantine condition demands a stronger condition for all possible commensurabilities of any order, ensuring that the frequency $\u{\omega}(\u{p})$ stays sufficiently far from resonances.

\subsection{Effective stability estimates}\label{sec:effective}
Under some conditions on the norm of the perturbing function of the system \equ{GeneralHamiltonian}, Nekhoroshev's theorem provides a confinement of the actions for exponentially {long time}. Following \cite{Poschel93}, we present in Section~\ref{sec:nonres} a version of Nekhoroshev's theorem for an $\alpha$, $K$ non-resonant domain. The statement of the theorem, while demanding some applicability conditions, leaves some freedom in the choice of the parameters. As a consequence, when making explicit estimates, a crucial point will be the optimization of such parameters, an issue that is analyzed in Section~\ref{sec:optimization} where an optimization procedure is proposed.

\subsubsection{Non-resonant effective stability estimates}\label{sec:nonres}
Consider the Hamiltonian \equ{GeneralHamiltonian}, 
which we assume to be real analytic in $D$, with $PD\subseteq\real^n$ open.
Let us consider a complex extension $V_{r_0}PD\times W_{s_0}\torus^n$, for $r_0, s_0\in\real_+$, where
\beqano 
V_{r_0}PD&=&\{\underline{p}\in\complex^n:\ dist(\underline{p},PD)<r_0 \}\nonumber\\
W_{s_0}\torus^n&=&\{\underline{q}\in\complex^n:\ Re(q_j)\in\torus,\ \ 
|Im(q_j)|<s_0,\ \ j=1,...,n \}\ . 
\eeqano 
For a function $f$ analytic in $V_{r_0}PD\times W_{s_0}\torus^n$, we introduce the norm 
$$
\|f\|_{r_0,s_0}=\sup_{\underline{p}\in V_{r_0}PD} \ \sum_{\underline{k}\in\integer^n}
|f_{\underline{k}}(\underline{p})|\ \exp\left(\norm{\underline{k}}_1s_0\right).
$$
If $E$ and $M$ are upper bounds on the norms, respectively, of the perturbation $\R(\u{p},\u{q})$ and the Hessian of $h(\u{p})$: 
\beq{bounds}
\norm{\R(\u{p},\u{q})}_{r_0,s_0}\leq E\quad ,\quad \sup_{\underline{p}\in V_{r_0}PD} \norm{\partial_{\underline{p}}^2 h(\underline{p})}_2\leq M, 
\eeq
where $\|\cdot\|_2$ is the Euclidean norm for the Hessian matrix, 
then the following statement holds (we refer the reader to \cite{Poschel93} for the proof). 

\begin{theorem}\label{thm:nonres}
Consider a Hamiltonian of the form \equ{GeneralHamiltonian} defined on a completely $\alpha$, $K$-nonresonant domain $D$. For some positive parameters $j$ and $\ell$, such that $1/j+1/\ell=1$, and $r$ such that
$$
r\leq \min({\alpha\over {j\,M\,K}},r_0)\ ,
$$
let $E$ be bounded by 
$$
E\leq{1\over {2^7\ell}}\ {{\alpha\, r}\over {K}}\ .
$$
Then, for any initial condition $(\underline{p}_0,\underline{q}_0)\in D$, the evolution of the system in the action space is bounded as  
$$
\|\underline{p}(t)-\underline{p}_0\|\leq r\qquad {\rm for\ all}\ |t|\leq T\equiv {{s_0r}\over {5E}}\ \exp\left(\frac{K s_0}{6}\right)\ .
$$
\end{theorem}

\subsubsection{Optimization of the parameters}\label{sec:optimization}

The application of Theorem~\ref{thm:nonres} requires a choice of the parameters $(r_0, s_0, \alpha, K, M, E, \ell)$ that, while necessary to satisfy the applicability conditions of Theorem~\ref{thm:nonres}, have some freedom. This means that for a given initial condition $(\u{p}_0,\u{q}_0)$, if there exists a choice of these parameters that satisfies the conditions of Theorem~\ref{thm:nonres}, then there exists also an optimal choice of the parameters that maximizes the stability time, though confining the actions to a (possibly) small neighborhood of the initial data. 

More specifically, if $\K\subseteq\zed^n_K/\left\{\u{0}\right\}$ is the set of Fourier indices appearing in the perturbing function $\R(\u{p},\u{q})$, one needs to choose the parameters 
\beq{conditions}
\left\{\begin{array}{l}
     \displaystyle E_{min}\equiv\norm{\R(\u{p}_0,\u{q}_0)}_{r_0,s_0}\leq E\leq \E \equiv{\frac{1}{2^7\ell}}\ \frac{\alpha\ r}{K}  \\
     ~\\
     \displaystyle\alpha\leq \alpha_{max}\equiv \inf_{\u{k}\in \mathbb{Z}^{n}_{K}}|\u{k}\cdot\u{\omega}(\u{p}_0)|\\
     ~\\
     \displaystyle M\geq M_{min}\equiv\sup_{\norm{\u{p}-\u{p}_0}\leq r_{0}}\left(\norm{\partial^2_{\u{p}}h(\u{p})}_2\right)\\
     ~\\
     \displaystyle r\leq r_{max}\equiv\min\left(\frac{\alpha}{MK}\left(1-\frac{1}{\ell}\right),r_0\right)\\
     ~\\
     K\geq \tilde{K} \quad\textrm{with } \quad\tilde{K}\equiv\sup_{\u{k}\in\K}\left(\norm{\u{k}}_1\right)
\end{array}\right. \begin{array}{l}
     \displaystyle\textrm{such that}\hspace{0.25cm}T\equiv {{s_0r}\over {5E}}\ \exp\left(\frac{K s_0}{6}\right)  \\
     \textrm{is maximized.}\\
\end{array}
\eeq
The optimal choice will always be such that $E=E_{min}$ and $r=r_{max}$, which implies the choice $M=M_{min}$ and  $\alpha=\alpha_{max}$. As a result, the choice of arbitrary parameters has been reduced to $(r_0,s_0,K,\ell)$. 

In most cases of physical interest, the norm of the perturbation $\norm{\R(\u{p}_0,\u{q}_0)}_{r_0,s_0}$ is the main limiting factor, since the perturbing function is not always of small magnitude. For this reason, we can simplify the optimization procedure by choosing the parameter $\ell$ such that it maximizes $\E$. Doing this, we are able to relax the condition on the norm, which allows us to apply Theorem~\ref{thm:nonres} more consistently and with greater freedom in the choice of the remaining parameters. In other words, while maximizing $\E$ is reducing the time $T$, it allows us to choose an optimal value for $s_0$, which increases $T$ exponentially. In the models we are going to analyze in Section~\ref{sec:results}, this is almost always the best choice and significantly reduces the number of computations that are needed for the optimization.

Considering the initial conditions $(\u{p}_0,\u{q}_0)$ and setting the parameter $r_0$ as input to initialize the maximization process, we are led to optimize the set of variables $(s_0,K)$. To this end, we implement an optimization process, where we compute the results for each value of $K$, starting from an initial guess and increasing $K$ as far as the results are improving, storing previous values of $K$ to be used whenever the code starts to diverge. Finally, for every choice of $K$, we compute $s_0$ so that it maximizes the value of $T$. More details on the optimization algorithm are given in Appendix~\ref{AppendixAlgorithm}.

\subsection{Perturbation theory}\label{sec:PertTheory}\label{sec:PM}
Although the upper bound, required by Theorem~\ref{thm:nonres}, on the norm of the perturbing function has been relaxed through a suitable choice of the parameter $\ell$ (See Section~\ref{sec:optimization}), it might {still not be enough} for the computation of stability estimates in physical systems. Instead, it is convenient to reduce the norm of the perturbing function by implementing perturbation theory, which is based on the construction of near-identity canonical transformations{,} as we are going to describe below. 

We consider the Hamiltonian \equ{GeneralHamiltonian}, defined on an $\alpha$, $K$ non-resonant modulo $\Lambda$ domain $D$. We split the perturbation into a non-resonant part $\R_{\underline{k}\in\zed_{K}^n/\Lambda}(\u{p},\u{q})$ and a resonant part $\R_{\underline{k}\in\Lambda}(\u{p},\u{q})$. Additionally, we separate the integrable Hamiltonian $h$ as $h(\u{p} )=\N(\u{p})+\Z(\u{p})$, namely the sum of a normal form $\N(\u{p})$ and an integrable correction $\Z(\u{p})$ with $\norm{\Z+\R_{\underline{k}\in\zed_{K}^n/\Lambda}}_{D}$ being sufficiently smaller than $\norm{\N}_{D}$.

Let $\L_\chi$ be the Lie operator, $\L_\chi=\{\cdot,\chi\}$ with $\{\cdot,\cdot\}$ denoting the Poisson brackets, and let $\exp(\L_\chi)$ be the Lie series operator
$$
\exp(\L_\chi)=\sum_{j=0}^\infty{1\over {j!}} \L_\chi^j\ .
$$
We consider a Lie canonical transformation of variables defined through a generating function $\chi=\chi(\underline{p}^{(1)},\underline{q}^{(1)})$:
\beq{coord}
(\underline{p},\underline{q})\mapsto(\underline{p}^{(1)},\underline{q}^{(1)})=(\exp(\L_{\chi(\underline{p}^{(1)},\underline{q}^{(1)})})\underline{p},\exp(\L_{\chi(\underline{p}^{(1)},\underline{q}^{(1)})})\underline{q})
\eeq
with back-transformation
\beq{backcoord}
(\underline{p}^{(1)},\underline{q}^{(1)})\mapsto(\underline{p},\underline{q})=(\exp(-\L_{\chi(\underline{p}^{(1)},\underline{q}^{(1)})})\underline{p}^{(1)},\exp(-\L_{\chi(\underline{p}^{(1)},\underline{q}^{(1)})})\underline{q}^{(1)})\ .
\eeq
The transformed Hamiltonian can be written as (see \cite{Laplata}, \cite{Ferraz})
$$
\mathcal{H}^{(1)}(\underline{p}^{(1)},\underline{q}^{(1)})=\exp(\L_{\chi(\underline{p}^{(1)},\underline{q}^{(1)})})\mathcal{H}(\underline{p}^{(1)},\underline{q}^{(1)})\ .
$$
As it is well known, the generating function is determined by solving the homological equation 
\beq{NFeq}
\L_{\chi(\underline{p}^{(1)},\underline{q}^{(1)})}\N(\underline{p}^{(1)})+\R_{\underline{k}\in\zed_{K}^n/\Lambda}(\underline{p}^{(1)},\underline{q}^{(1)})=0\ , 
\eeq
whose solution can be written as 
\beq{chi}
\chi(\underline{p}^{(1)},\underline{q}^{(1)})=\sum_{\underline{k}\in \zed_K^n\backslash\{\Lambda\cup\underline{0}\}} 
{{R_{\underline{k}}(\underline{p}^{(1)})}\over {i\underline{k}\cdot\underline{\omega}(\underline{p}^{(1)})}}\ e^{i \underline{k}\cdot\underline{q}^{(1)}}\ ,
\eeq
which is well defined for all $(\u{p}^{(1)},\u{q}^{(1)})\in D^{(1)}$ with $D^{(1)}$ an $\alpha$, $K$ non-resonant domain modulo $\Lambda$.
Then, the new Hamiltonian is given by
\beqano 
\H^{(1)}(\underline{p}^{(1)},\underline{q}^{(1)})
&=&\sum_{j=0}^\infty {1\over {j!}} \left(\L_{\chi(\underline{p}^{(1)},\underline{q}^{(1)})}^j \N(\underline{p}^{(1)})+\L_{\chi(\underline{p}^{(1)},\underline{q}^{(1)})}^j \Z(\underline{p}^{(1)})+\L_{\chi(\underline{p}^{(1)},\underline{q}^{(1)})}^j \R(\underline{p}^{(1)},\underline{q}^{(1)})\right)\ .
\eeqano 
Using \equ{NFeq}, we obtain 
\beqano 
\H^{(1)}(\underline{p}^{(1)},\underline{q}^{(1)})&=&\N(\underline{p}^{(1)})+\Z(\underline{p}^{(1)})+\R_{\underline{k}\in\Lambda}(\underline{p}^{(1)},\underline{q}^{(1)})+\R'(\underline{p}^{(1)},\underline{q}^{(1)})
\eeqano
with 
\beqano \R'(\underline{p}^{(1)},\underline{q}^{(1)})&=&\sum_{j=1}^\infty \Bigg({j\over {(j+1)!}} \L_{\chi(\underline{p}^{(1)},\underline{q}^{(1)})}^j \R_{\underline{k}\in\zed_{K}^n/\Lambda}(\underline{p}^{(1)},\underline{q}^{(1)})\nonumber\\&&\hspace{1.5cm}+{1\over {j!}} \left(\L_{\chi(\underline{p}^{(1)},\underline{q}^{(1)})}^j \Z(\underline{p}^{(1)})+\L_{\chi(\underline{p}^{(1)},\underline{q}^{(1)})}^j\R_{\underline{k}\in\Lambda}(\underline{p}^{(1)},\underline{q}^{(1)})\right)\Bigg)\ .
\eeqano 
We notice that, in the above expression, first-order perturbation terms have disappeared (see Appendix \ref{app:perturbstep} for more details on the computation of $\R'$). Moreover, the remainder is now of second order; in fact, recalling \equ{chi}, its norm is $\norm{\left\{\R_{\underline{k}\in\zed_{K}^n/\Lambda}(\underline{p}^{(1)},\underline{q}^{(1)}),\chi(\underline{p}^{(1)},\underline{q}^{(1)})\right\}}_{D^{(1)}}$\\${\sim} \norm{\R_{\underline{k}\in\zed_{K}^n/\Lambda}(\underline{p}^{(1)},\underline{q}^{(1)})}^2_{D^{(1)}}$.

Finally, the higher-order corrections of the integrable part are given by 
$$
\Z^{(1)}(\underline{p}^{(1)})=\Z(\underline{p}^{(1)})+\left<\R'(\underline{p}^{(1)},\underline{q}^{(1)})\right>_{\underline{k}\in\zed_{K}^n/\Lambda},
$$
where $\left<\cdot\right>_{\underline{k}\in\zed_{K}^n/\Lambda}$ denotes the average over the non-resonant angles, and the remainder function (in the new coordinates) takes the form
$$
\R^{(1)}(\underline{p}^{(1)},\underline{q}^{(1)})=\R_{\underline{k}\in\Lambda}(\underline{p}^{(1)},\underline{q}^{(1)})+\R'(\underline{p}^{(1)},\underline{q}^{(1)})-\left<\R'(\underline{p}^{(1)},\underline{q}^{(1)})\right>_{\underline{k}\in\zed_{K}^n/\Lambda}\ . 
$$
Of course, this procedure can be iterated to further reduce the norm of the perturbing function, up to some optimal order, before the accumulation of small divisors starts increasing the norm of the remainder (see \cite{Laplata}, \cite{Efthymiopoulos_Giorgilli_Contopoulos_2004}).

After applying Theorem~\ref{thm:nonres} to the transformed Hamiltonian $\H^{(J)}(\u{p}^{(J)},\u{q}^{(J)})$ at step $J$, we back-transform the initial conditions and we confine the actions in terms of the initial coordinates through \equ{backcoord}. 

We will refer to the resonances that correspond to the Fourier indices associated with $\R(\underline{p},\underline{q})$, as \sl main \rm resonances. \sl Secondary \rm resonances of $J^{\textrm{th}}$ order correspond to the Fourier indices associated with $\R^{(J)}(\underline{p}^{(J)},\underline{q}^{(J)})$, which represents the perturbing function after $J$ perturbative steps.

\subsection{Sequences of invariant tori}\label{sec:toripo}
While the implementation of some perturbative steps reduces the size of the perturbing function, it also adds to the new Hamiltonian a significant number of terms, represented by higher-order corrections. These additional terms introduce a large number of new harmonics into the Hamiltonian. As a consequence, studying the stability times in the vicinity of resonances becomes challenging, as these new harmonics are now interfering with the computations.

In order to reduce this interference as much as possible, an appropriate choice of the initial conditions, around the resonance under study, can be used. More specifically, we construct sequences of irrational Diophantine frequencies converging to the frequencies defining the resonances. Since the procedures to construct such sequences are different for the 1D and 2D cases (like those studied in Section~\ref{sec:results}), we treat each case in a separate subsection. 

\subsubsection{Diophantine frequencies for a 1D time-dependent Hamiltonian}\label{sec:1D}
In the case of a non-autonomous 1D system, we have ${\u\omega}(\u{p})=(\omega_1(p_1),\nu)=(\partial_{p_1}h(p_{1}),\nu)$. In normalized units, we can set $\nu=1$. Therefore, in the proximity of a resonance of order  $(k_{1},k_{2})$, we seek sequences of irrational frequencies that approximate the resonant frequency $\omega_1=-k_2/k_1$. In other words, we look for invariant tori converging to the periodic orbit.

Such irrational sequences, approximating the periodic orbit from above and below, can be defined as 
\beq{SOirr}
\Gamma_{z,s}^{(k_1,k_2)}=-{k_2\over k_1}-{s\over {z+\gamma}}\ ,\qquad 
\Delta_{z,s}^{(k_1,k_2)}=-{k_2\over k_1}+{s\over {z+\gamma}}
\eeq
for $z\in\integer\backslash\{0\}$, $s\in\real_+$, $\gamma$ being the golden ratio, $\gamma=(\sqrt{5}-1)/ 2$. It is clear that the sequences $\Gamma_{z,s}^{(k_1,k_2)}$, $\Delta_{z,s}^{(m,n)}$ converge to $-k_2/k_1$ as $z$ increases.
 Besides, they satisfy the Diophantine condition (Definition~\ref{diophcond}) for rational values of $s$ as shown by the following result. 

\begin{proposition}
Let $s=d/w$ with $d,w\in\integer\backslash\{\u{0}\}$, then the numbers  ${\omega_1=}\Gamma_{z,s}^{(k_1,k_2)}$ {or} ${\omega_1=}\Delta_{z,s}^{(k_1,k_2)}$ defined in \equ{SOirr} satisfy the inequality
\beq{weakdioph}
|k_1\omega_1+k_2|\geq\frac{C}{|k_1|},\quad \forall (k_1,k_2)\in\integer^2\backslash\{\u{0}\}\ .
\eeq
\end{proposition}

Before giving the proof, we remark that the Diophantine condition in Definition~\ref{diophcond} is different from \equ{weakdioph} to which we refer as the \sl strong \rm Diophantine condition. 

\begin{proof}
We consider only $\Gamma_{z,s}^{(k_1,k_2)}$, since the proof for $\Delta_{z,s}^{(k_1,k_2)}$ is similar. We can write $\Gamma_{z,s}^{(k_1,k_2)}$ as 
$$
\Gamma_{z,s}^{(k_1,k_2)}={{I_1+I_2\sqrt{5}}\over I_3}
$$
with $I_1$, $I_2$, $I_3$ integers given by 
\beqano 
I_1&=&2 \left(z^2-z-1\right) (-k_2) w+(1-2 z) k_1 d\nonumber\\
I_2&=&k_1 d\nonumber\\
I_3&=&2 \left(z^2-z-1\right) k_1 w\ . 
\eeqano 
Then, $\Gamma_{z,s}^{(k_1,k_2)}$ is a solution of the equation 
$$
I_3(x-\Gamma_{z,s}^{(k_1,k_2)})(I_3 x-I_1+I_2 \sqrt{5})=0\ , 
$$
that we can write as 
$$
A_2 x^2+A_1 x+A_0=0\ . 
$$
Hence, $\Gamma_{z,s}^{(k_1,k_2)}$ is the solution of a second-order polynomial equation with integer coefficients, since 
\beqano 
A_2&=&4 \left(-z^2+z+1\right)^2 k_1^2 w^2\nonumber\\
A_1&=&-4 \left(z^2-z-1\right) k_1\, w \left(2 \left(z^2-z-1\right) (-k_2)\, w+k_1 (d-2 z \,d)\right)\nonumber\\
A_0&=&4 \left(z^2-z-1\right) \left(\left(z^2-z-1\right) (-k_2)^2 w^2+(1-2 z) (-k_2)\, k_1\, d\, w+k_1^2 d^2\right)\ . 
\eeqano 
By a theorem by Liouville (\cite{Khinchin}), any irrational algebraic number of degree\footnote{Namely, solution of an irreducible polynomial equation of order 2 with integer coefficients.} two satisfies the Diophantine condition. 
\end{proof}

We remark that, besides the Diophantine condition, KAM theory requires that the unperturbed Hamiltonian satisfies Kolmogorov's non-degeneracy condition, which for a 1D time-dependent Hamiltonian, amounts to asking that the second derivative with respect to the action of the integrable part is different from zero; this condition will be satisfied by the spin-orbit problem that we will consider in Section~\ref{sec:SO}. Given that the phase space associated with a non-autonomous one-dimensional Hamiltonian has dimension 3, the 2D KAM tori separate the phase space into invariant regions. As a consequence, the existence of tori with frequencies $\Gamma_{z,s}^{(k_1,k_2)}$ and  $\Delta_{z,s}^{(k_1,k_2)}$ as in \equ{SOirr} guarantees the confinement of the trajectories, in between these tori, for an infinite time. 

Instead, in the case of the 2D non-autonomous system associated {with} the spin-spin-orbit problem in Section~\ref{sec:SS} the associated phase space has dimension 5, and therefore, the invariant tori can no longer separate the phase space, thus rendering stability estimates more effective.

\subsubsection{Diophantine frequencies for a 2D time-dependent Hamiltonian}\label{sec:2D}
Let us consider a non-autonomous 2D Hamiltonian. Similarly to the 1D case, we construct sequences of irrational frequencies that converge to a given resonance, or rather to an intersection of resonances, which are defined as follows. 

An intersection of two lower-dimensional resonances of order $(k_1,k_{3,1})$, $(k_2,k_{3,2})$ for $k_1,k_2,k_{3,1},k_{3,2}\in\integer\backslash\{0\}$ is defined by the set of equations
\beq{res1Dbis}
k_1\omega_1(\u p_0)+k_{3,1}\ =\ 0\ ,\qquad k_2\omega_2(\u p_0)+k_{3,2}=0\ .
\eeq
Setting $\u{\omega}(\u p)=(\omega_1(\u p),\omega_2(\u p),1)$, the condition $\u{\omega}(\u p)\cdot \u k=0$ with $\u k=(k_1,k_2,k_3)$ and $k_3=k_{3,1}+k_{3,2}$, represents an intersection of the resonances of order  $(k_1,0,k_{3,1})$ and $(0,k_2,k_{3,2})$. 

\vskip.1in 

Following \cite{CellettiFL04}, in the general $n$-dimensional case, we define the vector $\u{\omega}=(\omega_{1},\ldots,\omega_{n-1},{1}) \in \mathbb{R}^{n}$ such that:
\beq{diophevec}
\left(\begin{array}{c}
    \omega_{1}\\
    \vdots \\
    \omega_{n-1}\\
    {1}\end{array}\right)=\left(\begin{array}{cccc}
    b_{1} \\
    \vdots &  & \A & \\
    b_{n-1} &  &  & \\
    {1} & {0} & {\cdots} & {0}\end{array}\right)\left(\begin{array}{c}
    1\\
    \alpha\\
    \vdots\\
    \alpha^{n-1}\end{array}\right)\ , 
\eeq
where the vector $(b_{1},\ldots,b_{n-1})\in\real^{n-1}$, the $(n-1)\times (n-1)$ dimensional matrix $\A$ has rational coefficients $a_j$ with $\det \A\neq 0$, and $\alpha$ is a real algebraic number of degree $n$. The vectors defined by \equ{diophevec} satisfy the Diophantine condition (we refer to \cite{CellettiFL04} for the proof).

For $n=3$, which encompasses the case of a 2D non-autonomous system, we choose $\alpha$ as the solution of the following real algebraic equation of degree 3: 
\beq{algnumb}
\alpha^3+\alpha^2-1=0\ . 
\eeq
The number $s=1/\alpha$ is the smallest Pisot-Vijayaraghavan (PV) number of degree 3 {(\cite{Cassels})}, whose value is about equal to $s=1.324716$.

For an intersection of resonances of {order} $(k_{1},0,k_{3,1})$, $(0,k_{2},k_{3,2})$, according to \equ{res1Dbis} we introduce the sequence of Diophantine vectors:
\beq{freqvecdiophss}
\underline{\omega}^{(k_{1},k_{3,1}),(k_{2},k_{3,2})}_{z}=(\omega_{1,z}^{(k_{1},k_{3,1})},\omega_{2,z}^{(k_{2},k_{3,2})},1)\ , \qquad k\in \mathbb{Z} \setminus \{0\}
\eeq
with 
\begin{align}
\omega_{1,z}^{(k_{1},k_{3,1})} &= -\frac{k_{3,1}}{k_{1}} \pm \left( \frac{\tilde{a}_{1}}{z}\alpha + \frac{\tilde{a}_{2}}{z}\alpha^{2} \right) \notag \\
\omega_{2,z}^{(k_{2},k_{3,2})} &= -\frac{k_{3,2}}{k_{2}} \pm \left( \frac{\tilde{a}_{3}}{z}\alpha + \frac{\tilde{a}_{4}}{z}\alpha^{2} \right)\ , 
\label{diophvectors}
\end{align}
where $\alpha$ was defined through \equ{algnumb} and $\tilde{a}_{j}$ are rational numbers properly chosen.

For each value of $z$, following the notation in \equ{diophevec}, we have:
\beqno
a_{j}=\frac{\tilde{a}_{j}}{z}\quad\text{for}\quad j=1,\dots,4, \qquad \u b=(-\frac{k_{3,1}}{k_{1}},-\frac{k_{3,2}}{k_{2}})\ ,
\eeqno 
where $a_{j}$ are the rational coefficients of $\A$. As $z$ tends to infinity, the Diophantine frequencies $\u{\omega}^{(k_{1},k_{3,1}),(k_{2},k_{3,2})}_{z}$ converge to $(b_1,b_2,1)$ from above or below.

\section{An application to problems of rotational dynamics in Celestial Mechanics}\label{sec:results}

In this {section}, we consider two problems of rotational dynamics in Celestial Mechanics, precisely the spin-orbit model and the spin-spin-orbit model (Section~\ref{sec:model}), which are described by a 1D time-dependent and a 2D time-dependent Hamiltonian, respectively. We present the results of the algorithm described in Section~\ref{sec:optimization} when applied to the spin-orbit model (Section~\ref{sec:SOresults}) and to the spin-spin-orbit problem (Section~\ref{sec:SSresults}), presenting the results of the action confinement and the stability time.

\subsection{The spin-orbit and spin-spin-orbit models}\label{sec:model}

Let us consider a system of two rigid ellipsoidal bodies $\S_1$ and $\S_2$, with masses $M_{S_1}$ and $M_{S_2}$, and semi-axes $a_1,b_1,c_1$ and $a_2,b_2,c_2$, respectively. We assume that the two bodies move on Keplerian ellipses around their barycenter. We set $(r(t),f(t))$ to be the radial distance and the true anomaly of $S_1$ as it orbits around $S_2$. For simplicity, we will further assume that the spin axes of both bodies are perpendicular to the orbital plane and coincide with their semi-axes $c_{1}$ and $c_2$.

Let $\underline{q}=(q_1,q_2)$ be the corresponding spin angles of the two bodies with $\underline{p}=(p_1,p_2)$ the associated conjugate momenta. If $\{I^{(k)}_i\}^{k=1,2}_{i=a,b,c}$ are the the principal moments of inertia of each body ($k=1,2$) and each direction ($i=a,b,c$), then we can express the Hamiltonian describing the dynamics of $\S_1$ and $\S_2$ as 
\beq{GeneralHam}
\mathcal{H}(\underline{q},\underline{p},r(t),f(t))= \frac{p_1^2}{2I^{(1)}_c}+\frac{p_2^2}{2I^{(2)}_c}+V(r(t),f(t),\underline{q}),
\eeq
where $\mu=G(M_{S_1}+M_{S_2})$ and $G$ is the gravitational constant. Following  \cite{Misquero}, \cite{CGM} (see also \cite{BCL25}), the potential $V(r(t),f(t),\underline{q})$ takes the form
\beq{GeneralPotential}
V(r(t),f(t),\underline{q})=\sum_{n=1}^{\infty}\mathcal{V}_{n}(r(t),f(t),\underline{q}) 
\eeq
with
\beq{GeneralPotential}
\mathcal{V}_{n}(r,f,\underline{q})=-\frac{m}{r}\sum_{\substack{l_1+l_2=n\\m_1=-l_1\\m_2=-l_2}}^{\substack{m_1=l_1\\m_2=l_2}}C^{l_1,l_2}_{m_1,m_2}\mathcal{Z}^{(1)}_{2l_1,2m_1}\mathcal{Z}^{(2)}_{2l_2,2m_2}\left(\frac{R_1}{r}\right)^{2l_1}\left(\frac{R_2}{r}\right)^{2l_2}\cos(2m_1(q_1-f)+2m_2(q_2-f))\ ,
\eeq
where $m=GM_{S_1}M_{S_2}$, and $R_1$, $R_2$ are the mean equatorial radii of the two bodies. In the above expression the constants $C^{l_1,l_2}_{m_1,m_2}$ are given by
\begin{multline}\label{eq:ConstantsPotential}
    C^{l_1,l_2}_{m_1,m_2}=\frac{(2(l_1 + l_2)-2(m_1 + m_2))!\ (2(l_1 + l_2)+2(m_1 + m_2))!}{((l_1 + l_2) - (m_1 + m_2))!\ ((l_1 + l_2) + (m_1 + m_2))!} \\\times\frac{(-1)^{(l_1 + l_2 - m_1 - m_2)}4^{-(l_1 + l_2)}}{
  \sqrt{(2l_1 - 2m_1)!\ (2l_2 - 2m_2)!\ (2l_1 + 2m_1)!\ (2l_2 + 2m_2)!}}
\end{multline}
and the constants $\mathcal{Z}^{(k)}_{l,m}$ are computed as
\begin{multline}\label{eq:StokesCoefs}
\mathcal{Z}^{(k)}_{l,m}=\frac{3}{4\pi}\frac{1}{R^l}\sqrt{\frac{(l-m)!}{(l+m)!}}\int_{(X,Y,Z)\in B(\underline{0},1)}P_{l,m}\left(\frac{c_{k}Z}{\sqrt{a_{k}^2X^2 + b_{k}^2Y^2 + c_{k}^2Z^2}}\right)\\\frac{Re\left((a_{k}X-ib_{k}Y)^m\right)(a_{k}^2X^2 + b_{k}^2Y^2 + c_{k}^2Z^2)^{l/2}}{(a_{k}^2X^2 + b_{k}^2Y^2)^{m/2}}dXdYdZ\ ,
\end{multline}
where $P_{l,m}(x)$ are the Legendre polynomials. We notice that the potential \equ{GeneralPotential} can be expressed as a function of the mean anomaly and the eccentricity using Lagrange's inversion theorem for Kepler's equation (\cite{Hagihara_1970_V1}, \cite{Whittaker_Watson_2021}).

Finally, we will consider the physical units to be such that $(M_{S_1}+M_{S_2})=1$, $G=a^3$, and $(I^{(1)}_{c}+I^{(2)}_{c})=1$.

\subsubsection{The spin-orbit model}\label{sec:SO}

Considering only the first term of the potential \equ{GeneralPotential} for $n=1$, we get\footnote{{For bodies which do not present a highly irregular shape, the approximation of the potential with $n=1$ represents a good starting model, since the coefficients $2c_k^2-a_k^2-b_k^2$ and $a_k^2-b_k^2$ are small.}}
\beq{SpiOrbitPotential}
V(r(t),f(t),\underline{q})=-\frac{m}{20 r(t)^3}\sum_{k=1,2}\Bigl[(2c_k^2-a_k^2-b_k^2) -3(a_k^2-b_k^2)\cos(2q_k-2f(t))\Bigl]\ .
\eeq
In this case, the spins $\underline{q}=(q_1,q_2)$ of the two bodies are separated, and therefore the Hamiltonian can be split into two 1D time-dependent Hamiltonians. 

For simplicity, let us consider the restricted case with $M_{S_1}\simeq 0$ and $c_k^2=(a_k^2+b_k^2)/2$ for {both bodies to eliminate} the constant terms in \equ{SpiOrbitPotential}. 

Under the above assumptions, the Hamiltonian for the body $\S_1$ is (see also \cite{Celletti2010})
\beq{SpiOrbitHam}
\mathcal{H}(q_1,p_1,t)= \frac{p_1^2}{2}-\frac{\varepsilon}{2}\frac{a^3}{r(t)^3}\cos(2q_1-2f(t))\ , 
\eeq
where we set $p_1=p_1/I^{(1)}_{c}$, $\varepsilon=3/2\, (I^{(1)}_b-I^{(1)}_a)/I^{(1)}_c=3M_{S_1}(a_1^2-b_1^2)/(10 I^{(1)}_c)$.
We express the Hamiltonian \equ{SpiOrbitHam} in terms of the mean anomaly by expanding in power series of the eccentricity and retaining terms up to degree 5; the explicit expression is given in Appendix~\ref{app:SO}. 

Finally, in order to be consistent with the literature, let us give the following definition of a spin-orbit resonance.
\begin{definition}\label{def4}
A spin-orbit resonance of {type} $k_2:k_1$ is a resonance of order $(k_1,-k_2)$ in the sense of Definition~\ref{resgeneral} for the Hamiltonian system \equ{SpiOrbitHam}.
\end{definition}
This definition physically implies that the body $S_1$ makes $k_1$ rotations within $k_2$ orbital revolutions.

\subsubsection{The spin-spin-orbit model}\label{sec:SS}
The interaction between the spin angles of the two bodies appears in the second term of the potential \equ{GeneralPotential}. More specifically, considering again the simpler case $c_k^2=(a_k^2+b_k^2)/2$, one finds 
\beqa{V2}
\mathcal{V}_{2}(r(t),f(t),\underline{q})&=&-\frac{m}{r(t)^5}\Biggl\{\sum_{k=1,2}\left(\frac{\varepsilon_{k}I^{(k)}_{c}}{M_{S_k}}\right)^2\left(\frac{5}{112}+\frac{25}{48}\cos(4q_k-4f(t))\right)\nonumber\\ 
&+& \left(\frac{\varepsilon_{1}I^{(1)}_{c}}{M_{S_1}}\right)\left(\frac{\varepsilon_{2}I^{(2)}_{c}}{M_{S_2}}\right)\bigg(\frac{35}{24} \cos(4 f(t) - 2 q_1 - 2 q_2)\nonumber\\
&+&\frac{1}{8} \cos(2 q_1 - 2 q_2)\bigg)\Biggl\} 
\eeqa
with $\varepsilon_k=3/2\ (I^{(k)}_b-I^{(k)}_a)/I^{(k)}_c=3M_{S_k}(a_k^2-b_k^2)/(10 I^{(k)}_c)$. As a result, the Hamiltonian for the spin-spin-orbit model is  given by 
\beq{SpinSpinHam}
\begin{array}{ll}
     \displaystyle\mathcal{H}(\underline{q},\underline{p},t)=&\displaystyle \frac{p_1^2}{2I^{(1)}_{c}}+\frac{p_2^2}{2I^{(2)}_{c}}
     \\~\\~&\displaystyle 
     -\frac{m}{r^3(t)}\sum_{k=1,2}\frac{\varepsilon_kI^{(k)}_{c}}{2M_{S_k}}\cos(2q_k-2f(t))\\~\\~&\displaystyle -\frac{m}{r(t)^5}\displaystyle\sum_{k=1,2}\left(\frac{\varepsilon_{k}I^{(k)}_{c}}{M_{S_k}}\right)^2\left(\frac{5}{112}+\frac{25}{48}\cos(4q_k-4f(t))\right)\\~\\~&\displaystyle -\frac{m}{r(t)^5} \left(\frac{\varepsilon_{1}I^{(1)}_{c}}{M_{S_1}}\right)\left(\frac{\varepsilon_{2}I^{(2)}_{c}}{M_{S_2}}\right)\left(\frac{35}{24} \cos(4 f(t) - 2 q_1 - 2 q_2) +\frac{1}{8} \cos(2 q_1 - 2 q_2)\right).
\end{array}
\eeq
Like for the spin-orbit case, the above Hamiltonian can be written as a function of the mean anomaly in a power series of the eccentricity by expressing it as a function of the eccentric anomaly and applying Lagrange's inversion theorem  (see Appendix~\ref{app:SS}).

Similarly to the spin-orbit case, we can give a definition for the spin-spin-orbit resonance.

\begin{definition}
A spin-spin-orbit resonance of {type} $k_3:k_1:k_2$ is a resonance of order $(k_1,k_2,-k_3)$ in the sense of Definition~\ref{resgeneral} for the Hamiltonian system \equ{SpinSpinHam}.
\end{definition}

Additionally, let us define a special case of a spin-spin-orbit resonance.

\begin{definition}
A spin-orbit/spin-orbit resonance of {type} $(k_{3,1}:k_1)_{S_1}$,$(k_{3,2}:k_2)_{S_2}$ occurs when there is an intersection between the spin-spin-orbit resonances of {types} $k_{3,1}:k_1:0$ and  $k_{3,2}:0:k_2$.
\end{definition}

In other words, a frequency vector $\u{\omega}=(\omega_1,\omega_2,1)$ is in a spin-orbit/ spin-orbit resonance $k_{3,1}:k_1$, $k_{3,2}:k_2$ when 
\beq{res1Dbisspin}
k_1\omega_1-k_{3,1}\ =\ 0\ ,\qquad k_2\omega_2-k_{3,2}=0\ . 
\eeq

\subsection{Stability estimates for the spin-orbit model}\label{sec:SOresults}

The stability estimates of Theorem~\ref{thm:nonres}, using the optimization algorithm described in Section~\ref{sec:optimization}, are implemented for the spin-orbit model (Section ~\ref{sec:SO}). We perform $J$ perturbation steps considering as normal form $\mathcal{N}(p_{1},p_{3})=p^2_{1}/2+p_{3}$ and neglecting terms of magnitude smaller than $10^{-20}$. We apply Theorem ~\ref{thm:nonres} to $\mathcal{H}^{(J)}$ using as frequency vector $\u\omega(\u p)=\nabla_{\u p}
h^{(J)}(\u p)$ and computing the norm of the perturbing function $\mathcal{R}^{(J)}$ following Appendix \ref{app:norm}. The results of the optimization procedure are then back-transformed to the original coordinates, using relation \equ{backcoord}.

To fix the grid of initial conditions, we consider non-resonant initial conditions  $(\underline{p}_{0},\underline{q}_{0})$, corresponding to a sequence of Diophantine frequencies like those introduced in Section~\ref{sec:1D}.

From Theorem ~\ref{thm:nonres}, we notice that the initial conditions on the angles do not influence the stability estimates and thus they can be set to zero; hence, we choose the following values:
$$(\underline{p}(0),\underline{q}(0))=(p_{1}(0),p_{3}(0),q_{1}(0),q_{3}(0))=(p_{1}(0),0,0,0)\ ,$$
where we have set to zero also the initial time $q_{3}(0)$ and its {conjugate} action $p_{3}(0)$.
 
For a given $k_2:k_1$ spin-orbit resonance, according to Definition ~\ref{def4} and \equ{SOirr} we generate a sequence of non-resonant initial conditions,  
corresponding to Diophantine frequency vectors such that
\beq{seqdioph}
\underline{\omega}(p_{1}(0))_{z}=(\Gamma_{z,s}^{(k_1,k_2)},1) \qquad\text{or}\qquad
\underline{\omega}(p_{1}(0))_{z}=(\Delta_{z,s}^{(k_1,k_2)},1)\  
\eeq
with an index $z\in\integer\backslash\ \{0\}$ and for a fixed value of $s$.
Therefore, for a given $k_{2}:k_{1}$ spin-orbit resonance, the invariant tori with frequency vectors \equ{seqdioph} bound the resonance from above or below in phase space (compare with \cite{Celletti2010}).

\begin{figure}[htbp]
\centering

\begin{minipage}{0.49\textwidth}
    \centering
    \includegraphics[width=\linewidth]{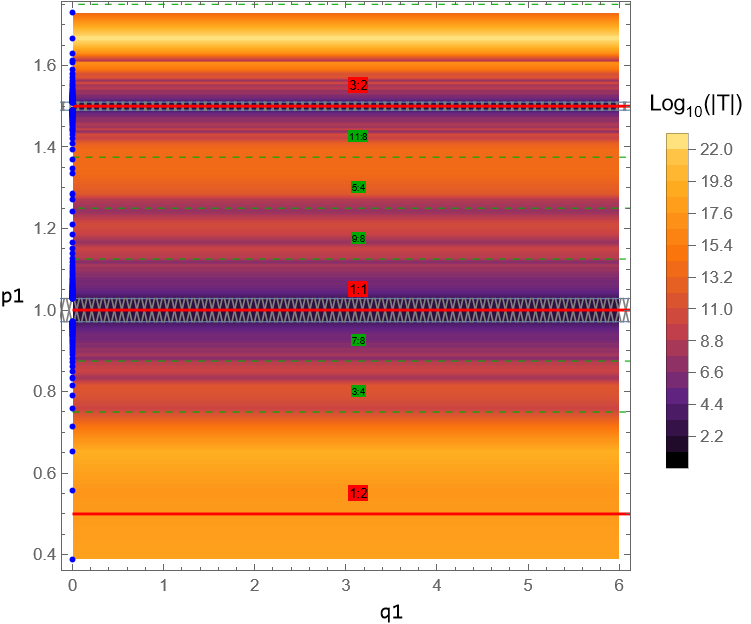}
    \par\footnotesize (a)
\end{minipage}%
\hfill
\begin{minipage}{0.49\textwidth}
    \centering
    \includegraphics[width=\linewidth]{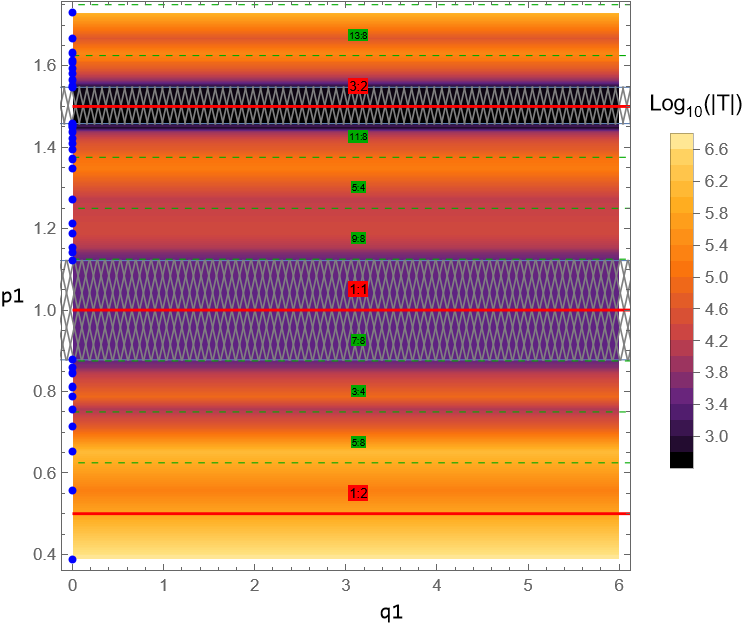}
    \par\footnotesize (b)
\end{minipage}

\vspace{1ex}

\begin{minipage}{0.49\textwidth}
    \centering
    \includegraphics[width=\linewidth]{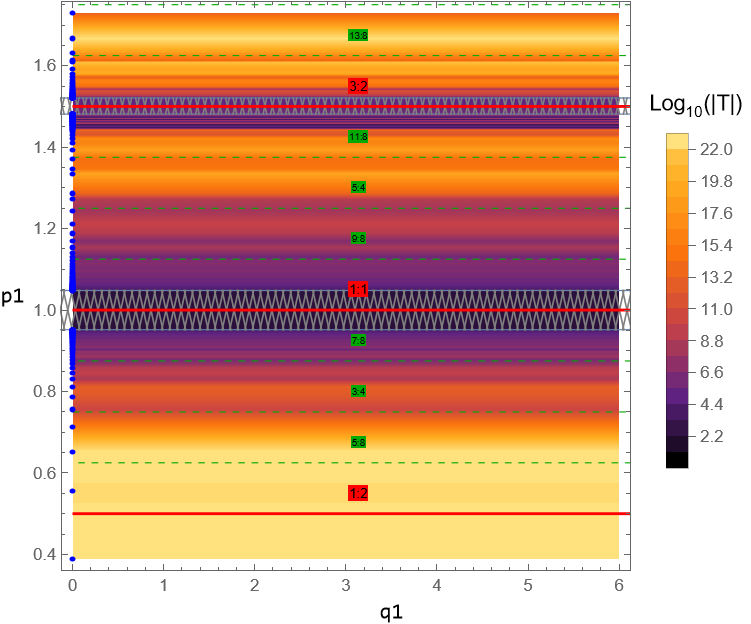}
    \par\footnotesize (c)
\end{minipage}%
\hfill
\begin{minipage}{0.49\textwidth}
    \centering
    \includegraphics[width=\linewidth]{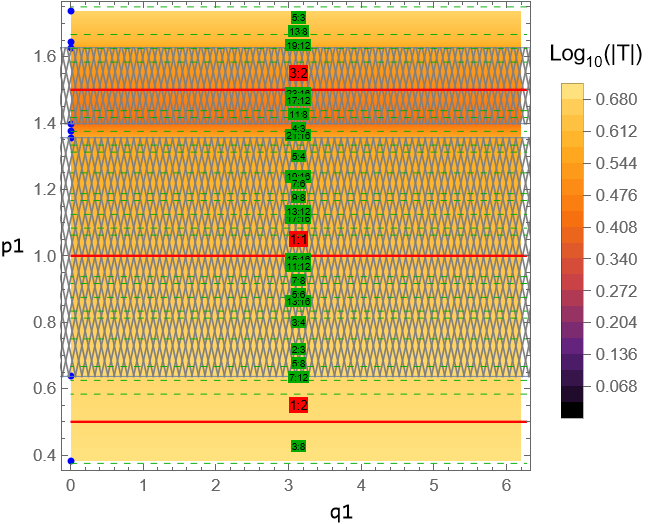}
    \par\footnotesize (d)
\end{minipage}

\vspace{1ex}
\caption{\small  Stability time for $z=2,\dots,100$ with $s=1.6$ ($1:1$ resonance) and $s=0.6$ ($3:2$ resonance) for (a) $\varepsilon=10^{-4}$ after two perturbative steps, (b) $\varepsilon=10^{-3}$ after two perturbative steps, (c) $\varepsilon=10^{-3}$ after three perturbative steps, (d) $\varepsilon=10^{-2}$ after three perturbative steps (d).}
\label{fig11}
\end{figure}

In Figure ~\ref{fig11}, we show the stability time $T$ (in a $\log_{10}$ color scale) on the $(q_{1},p_{1})$-plane. The stability time is obtained for the grid of initial conditions associated with two specific spin-orbit resonances, $1:1$ and $3:2$, considering different perturbing parameters $\varepsilon$, and applying a different number of perturbation steps $J$.

On the $(q_{1},p_{1})$-plane the initial conditions are indicated by the blue points located along the $p_{1}$-axis. The main resonances, precisely the $1:1$, $1:2${,} and $3:2$, are marked by horizontal red lines and the secondary resonances appear as green dashed lines. Finally, the gray cross-hatched regions are the resonant domains for which the optimization algorithm fails to find parameters that satisfy the applicability conditions of Theorem ~\ref{thm:nonres}.

Figures~\ref{fig11}(a) and~\ref{fig11}(b) show a decrease of the stability time as the $1:1$ and the $3:2$ spin-orbit resonances are approached. Similarly, secondary resonances of different orders, such as the $5:4$ or $3:4$ resonances, significantly affect the stability times of nearby initial conditions, resulting in the observed low values. In Figure 1, we can also see regions where stability results are affected, even if there are no main or secondary resonances nearby. {These regions are an indirect effect of the requirement that the domain is $\alpha$-$K$ non-resonant according to Definition~\ref{nonres}; this condition implies that the 
system has to be away from any commensurability of the form $\underline{k}\cdot\underline{\omega}$ for any order $\underline{k}$ such that $\norm{\underline{k}}_1\leq K$}.
Furthermore, Figure~\ref{fig11}(a) shows fewer secondary resonances compared to Figure~\ref{fig11}(b), since in Figure~\ref{fig11}(a), by choosing $\varepsilon=10^{-4}$ and applying two perturbative steps, some secondary resonances become negligible.

In Figure ~\ref{fig11}(b), as expected, having chosen a larger value of $\varepsilon$, the stability times are lower, either near or far from the resonances.

Setting \(\varepsilon=10^{-3}\), the algorithm does not provide stability estimates after a single perturbative step, since the norm of the remainder is too large to satisfy the assumptions of Theorem~\ref{thm:nonres}, while it works when we apply 2 or 3 perturbative steps (see Figures ~\ref{fig11}(b) and  ~\ref{fig11}(c)). Moreover, passing from \(J=2\) to \(J=3\), we obtain stability results, under the non-resonant condition, that are significantly closer to the resonant regions (see also Figure~\ref{fig22}). We stress that, for those initial values where the algorithm was already working, additional perturbative steps improve the results.

At \(\varepsilon=10^{-2}\), the algorithm only succeeds after \(J=3\) perturbative steps, see Figure ~\ref{fig11}(d), where we notice that the widths of the resonances become significantly large. The resonant domains for which the optimization algorithm fails to find parameters that satisfy the applicability conditions of Theorem ~\ref{thm:nonres} include several secondary resonances, especially around the 1:1 resonance.

\begin{figure}[htbp]
\centering

\begin{minipage}{0.49\textwidth}
    \centering
    \includegraphics[width=\linewidth]{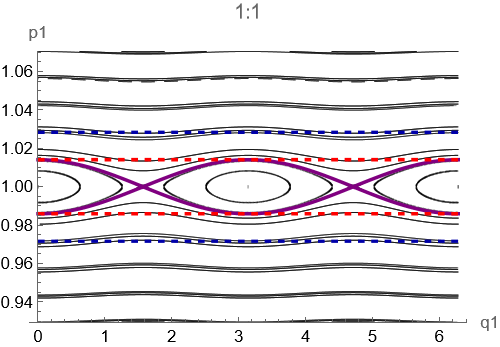}
    \par\footnotesize (a)
\end{minipage}
\hfill
\begin{minipage}{0.49\textwidth}
    \centering
    \includegraphics[width=\linewidth]{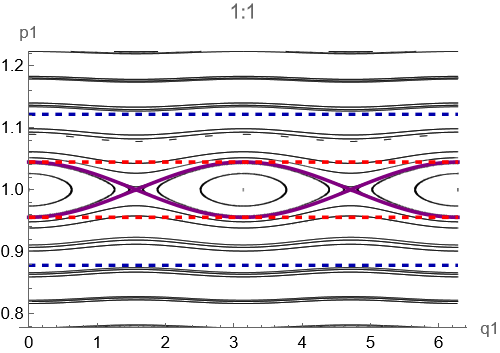}
    \par\footnotesize (b)
\end{minipage}

\vspace{1ex}

\begin{minipage}{0.49\textwidth}
    \centering
    \includegraphics[width=\linewidth]{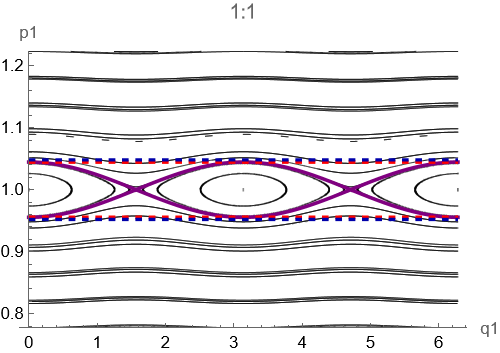}
    \par\footnotesize (c)
\end{minipage}%
\hfill
\begin{minipage}{0.49\textwidth}
    \centering
    \includegraphics[width=\linewidth]{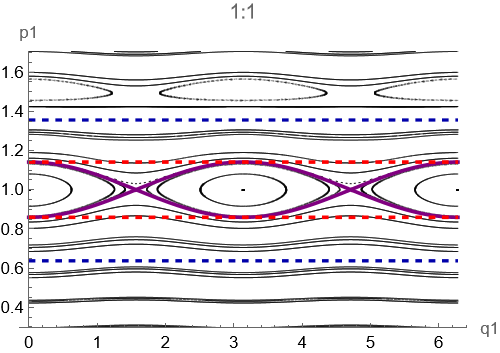}
    \par\footnotesize (d)
\end{minipage}

\vspace{1ex}
\caption{\small Comparison between the last value of $p_{1}$ for which the algorithm provides results (blue line) and a resonant normal form approximation of the width of the 1:1 resonance (red line) for (a) $\varepsilon=10^{-4}$ after two perturbative steps, (b) $\varepsilon=10^{-3}$ after two perturbative steps, (c) $\varepsilon=10^{-3}$ after three perturbative steps, (d) $\varepsilon=10^{-2}$ after three perturbative steps.}
\label{fig22}
\end{figure}

\begin{figure}[htbp]
\centering

\begin{minipage}{0.49\textwidth}
    \centering
    \includegraphics[width=\linewidth]{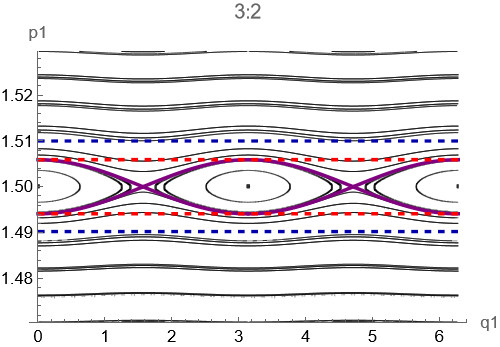}
    \par\footnotesize (a)
\end{minipage}%
\hfill
\begin{minipage}{0.49\textwidth}
    \centering
    \includegraphics[width=\linewidth]{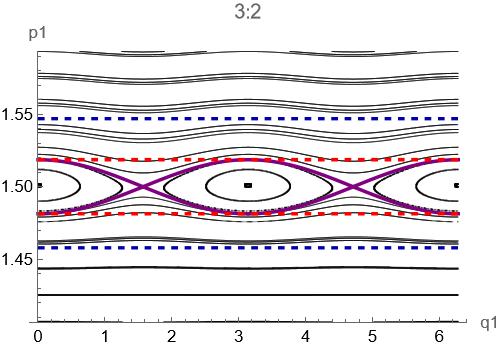}
    \par\footnotesize (b)
\end{minipage}

\vspace{1ex}

\begin{minipage}{0.49\textwidth}
    \centering
    \includegraphics[width=\linewidth]{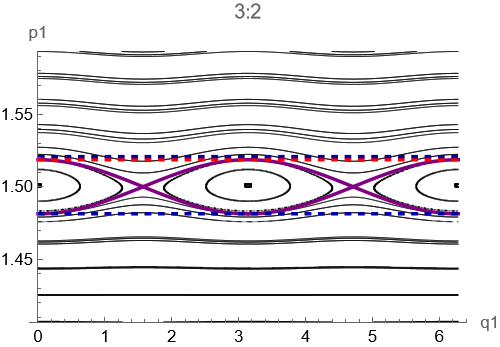}
    \par\footnotesize (c)
\end{minipage}%
\hfill
\begin{minipage}{0.49\textwidth}
    \centering
    \includegraphics[width=\linewidth]{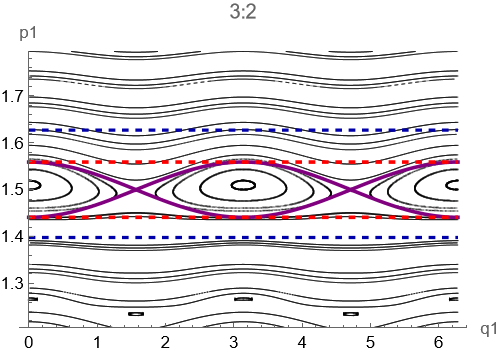}
    \par\footnotesize (d)
\end{minipage}

\vspace{1ex}
\caption{\small Comparison between the last value of $p_{1}$ for which the algorithm provides results (blue line) and a resonant normal form approximation of the width of the 3:2 resonance (red line) for (a) $\varepsilon=10^{-4}$ after two perturbative step, (b) $\varepsilon=10^{-3}$ after two perturbative steps, (c) $\varepsilon=10^{-3}$ after three perturbative steps, (d) $\varepsilon=10^{-2}$ after three perturbative steps.}
\label{fig33}
\end{figure}

Finally, in Figures~\ref{fig22}  and ~\ref{fig33} we compare, respectively, the last value of $p_{1}$, back transformed, for which the algorithm still provides results (blue line) with a resonant normal form approximation of the width (red line) of the $1:1$ (Figure ~\ref{fig22}) and $3:2$ (Figure ~\ref{fig33}) spin-orbit resonances. Panels (b) and (c) provide the results for the same value $\eps=10^{-3}$ but, respectively, after two and three perturbative steps, thus showing that implementing perturbation theory combined with the estimates of Theorem~\ref{thm:nonres} leads to a significant improvement in terms of proximity to the resonances.

\subsection{Stability estimates for the spin-spin-orbit model}\label{sec:SSresults}
In this Section, we apply the optimization procedure to the spin-spin-orbit model described in  Section ~\ref{sec:SS}. 
As for the spin-orbit model{,} we implement $J$ perturbation steps considering, in this case, as normal form $\mathcal{N}(p_{1},p_{2},p_{3})=p^2_{1}/2+{p^2_{2}/2}+p_{3}$. We keep terms of order smaller than $(\varepsilon_{1}^{a}\varepsilon_{2}^{b})$, with $a+b=5$, and magnitude bigger than   $10^{-20}$ while neglecting the rest. We apply the optimization algorithm (Section ~\ref{sec:optimization}) to $\mathcal{H}^{(J)}$ using as frequency vector, $\u\omega(\u{p})=\nabla_{\u p}
h^{(J)}(\u p)$ and computing the norm of the perturbing function $\mathcal{R}^{(J)}$ following Appendix \ref{app:norm}. Finally, we show the results back-transformed to the original coordinates.

As in the spin-orbit model, the angles, $(q_{1}(0)$, $q_{2}(0))$, the initial time $q_{3}(0)$, and the associated dummy action $p_{3}(0)$ can be set to zero. Therefore, the choice of the initial values is reduced to selecting $(p_{1}(0)$, $p_{2}(0))$.

For a given spin-orbit/spin-orbit resonance of {type} $(k_{3,1}:k_{1})_{S_1},(k_{3,2}:k_{2})_{S_2}$, according to \equ{diophvectors} and varying suitably the coefficients of the matrix $\A$ in \equ{diophevec}, we construct a radial grid of initial conditions $(p_{1}(0),p_{2}(0))_{z}$, such that:
\beq{omegaspispin}
\underline{\omega}(p_{1}(0),p_{2}(0))_{z}=(\omega_{1,z}^{(k_{1},k_{3,1})},\omega_{2,z}^{(k_{2},k_{3,2})},1) 
\eeq
with the components given by 
$$
\omega_{1}(p_{1}(0),p_{2}(0))_{z}\xrightarrow{z \to \infty} -\frac{k_{3,1}}{k_{1}}, \qquad \omega_{2}(p_{1}(0),p_{2}(0))_{z} \xrightarrow{z \to \infty}-\frac{k_{3,2}}{k_{2}}\ ,\qquad z\in \mathbb{Z^{+}}\setminus\{0\}\ . 
$$
\begin{figure}[htbp]
\centering

\begin{minipage}{0.46\textwidth}
    \centering
    
    \includegraphics[width=\linewidth]{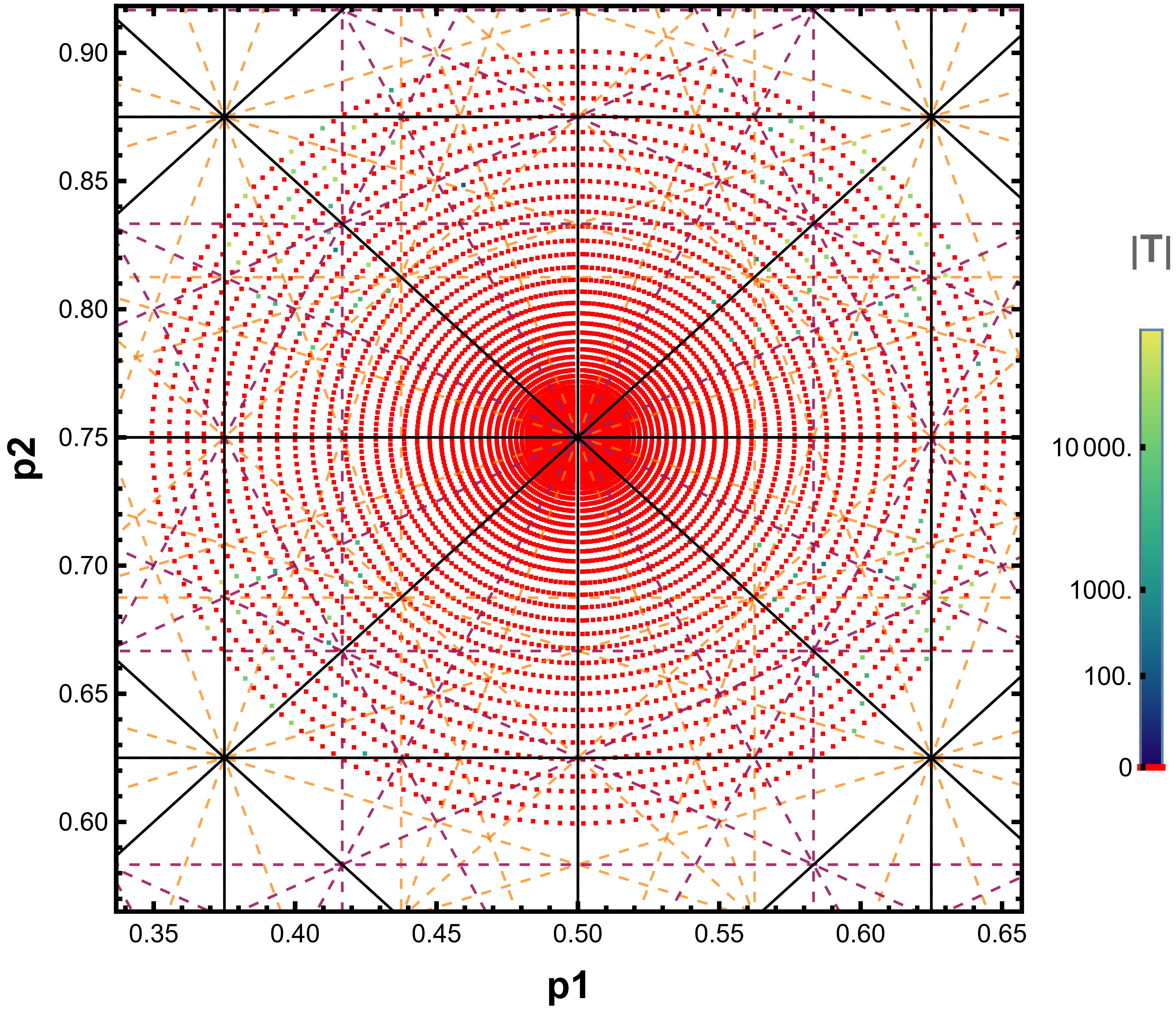}
    \par\footnotesize (a)
\end{minipage}%
\hfill
\begin{minipage}{0.46\textwidth}
    \centering
    
    \includegraphics[width=\linewidth]{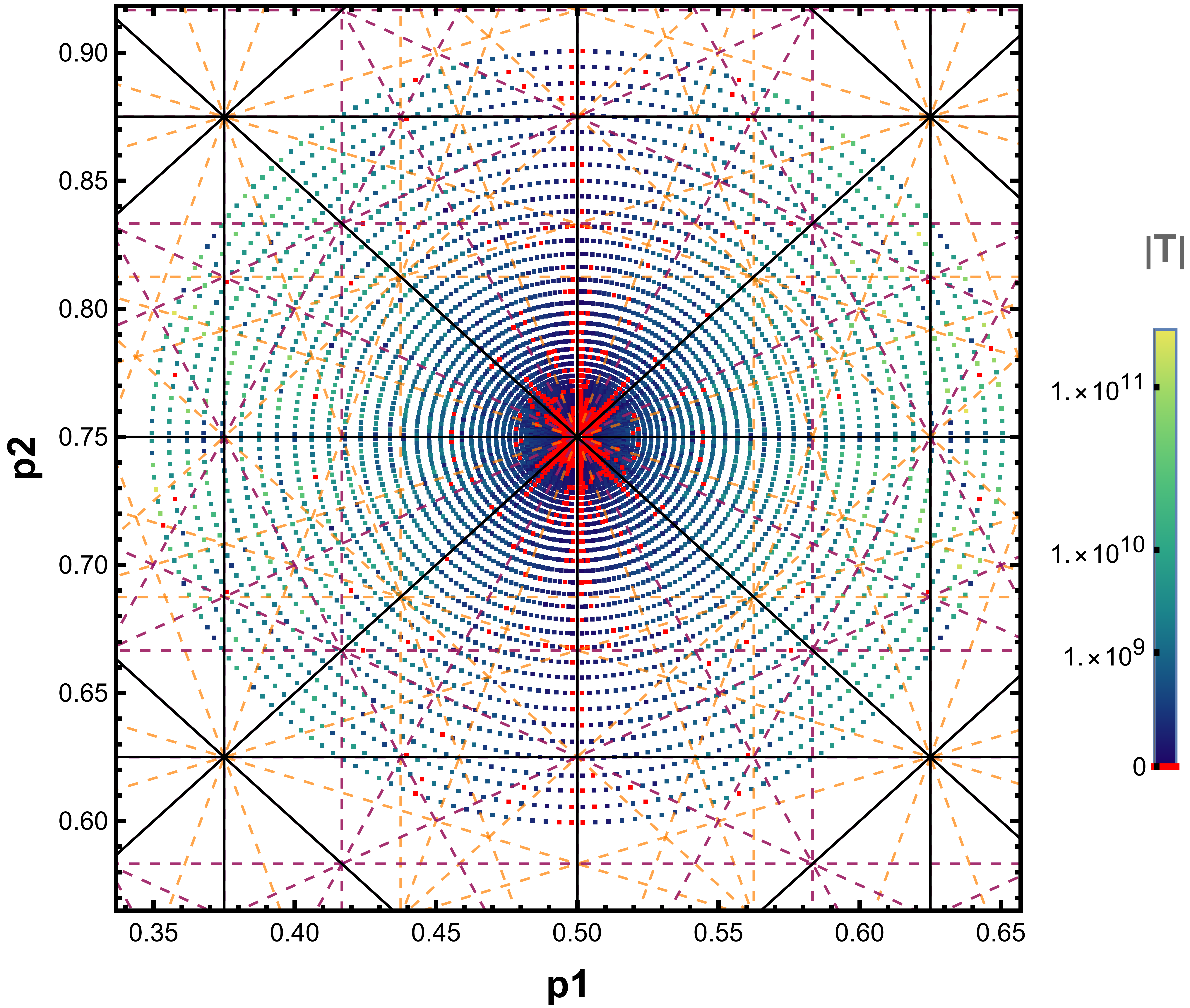}
    \par\footnotesize (b)
\end{minipage}

\vspace{1ex}

\begin{minipage}{0.46\textwidth}
    \centering
    \includegraphics[width=\linewidth]{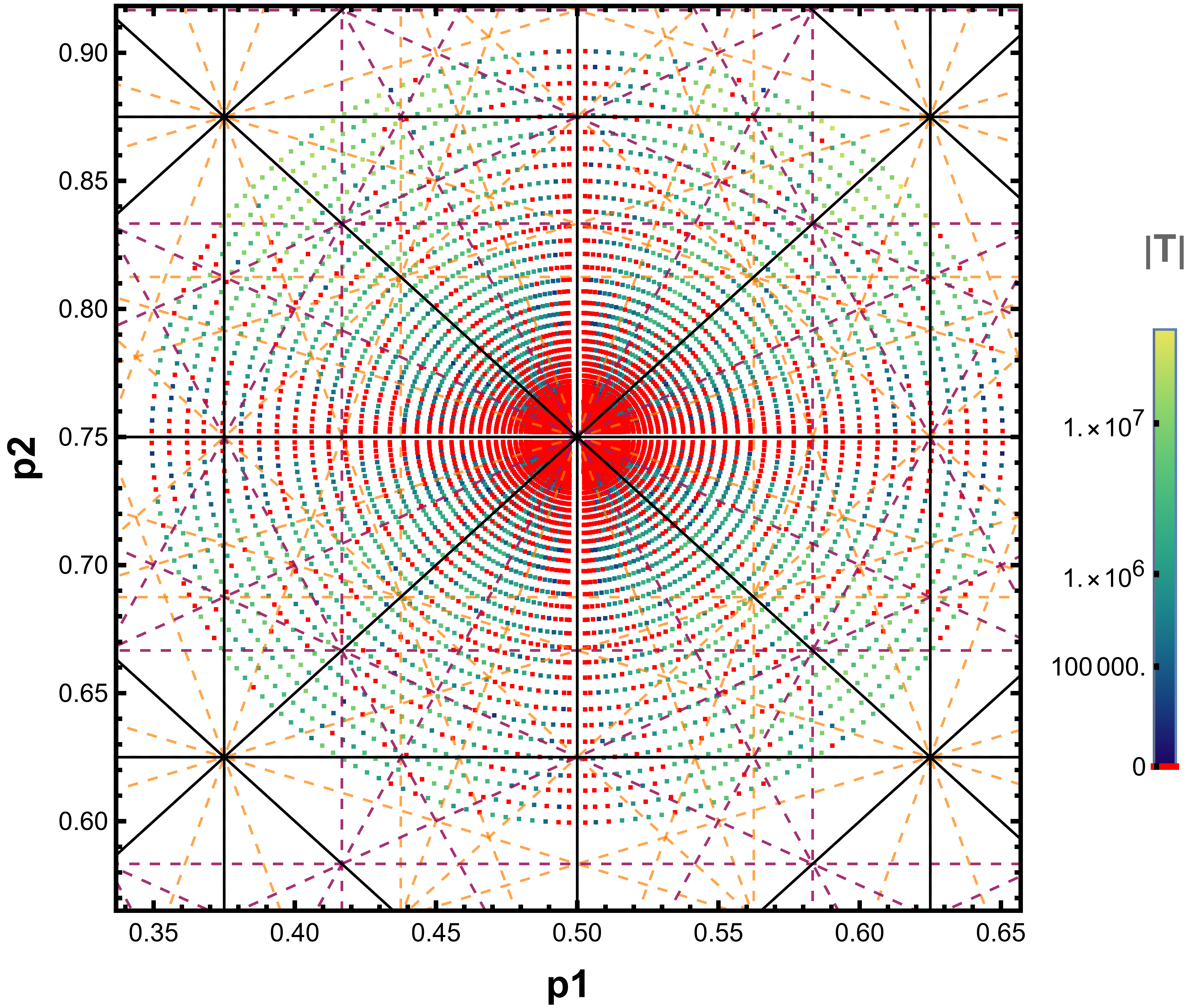}
    \par\footnotesize (c)
\end{minipage}
\hfill
\begin{minipage}{0.46\textwidth}
    \centering
    \includegraphics[width=\linewidth]{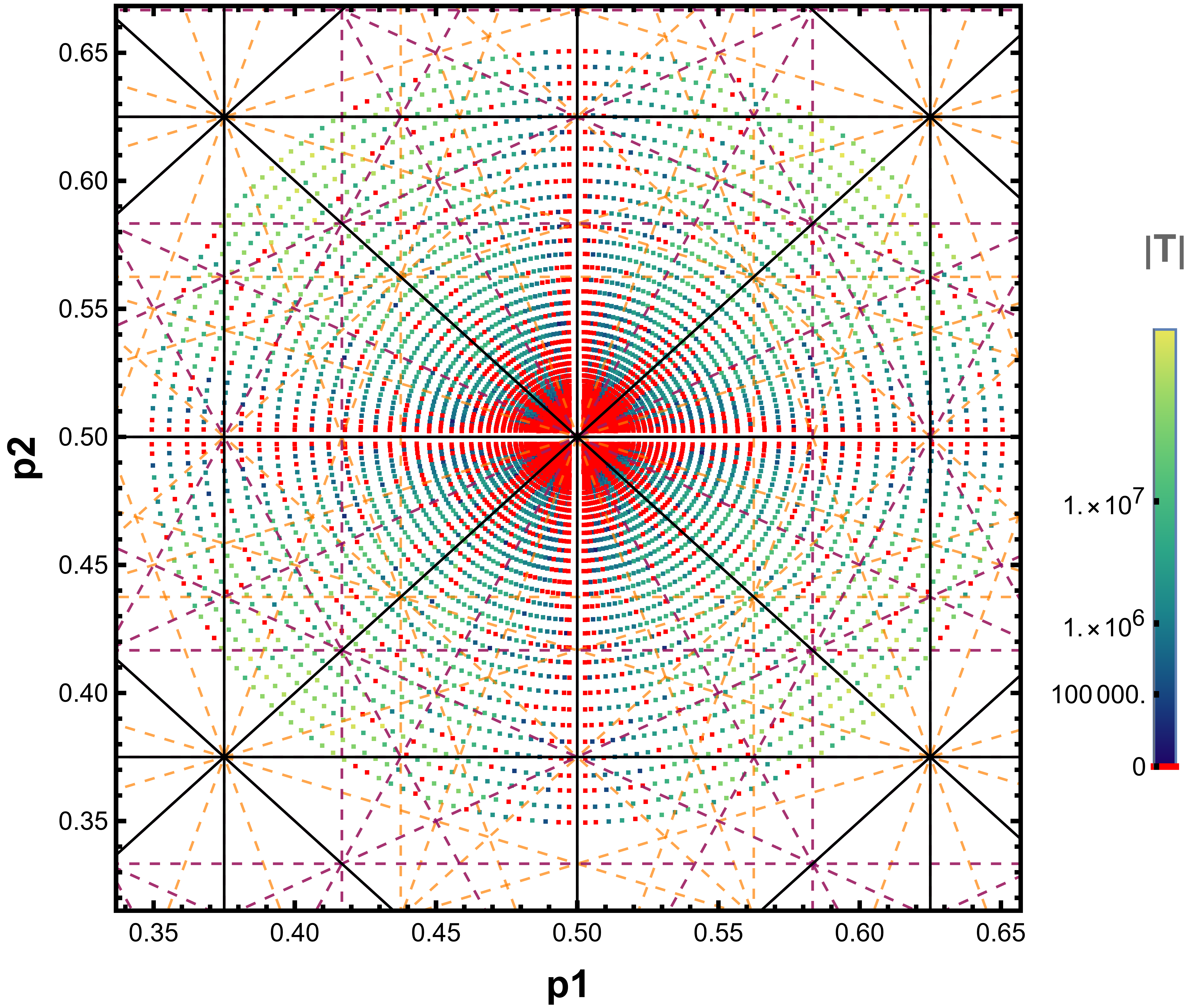}
    \par\footnotesize (d)
\end{minipage}
\vspace{-0.3cm}
\caption{\small Stability time around the spin-spin-orbit resonance $(1:1)_{S_1},(3:2)_{S_2}$, for $\varepsilon_{1}=\varepsilon_{2}=10^{-5}$ after: (a) one perturbative step, (b) two perturbative steps. 
Stability time around the spin-spin-orbit resonance: (c) $(1:1)_{S_1},(3:2)_{S_2}$, for $\varepsilon_{1}=3\cdot 10^{-5}$ and  $\varepsilon_{2}=10^{-4}$ after two perturbative steps, (d) $(1:1)_{S_1},(1:1)_{S_2}$, for $\varepsilon_{1}=\varepsilon_{2}=3\cdot10^{-5}$ after two perturbative steps.}
\label{fig44}
\end{figure}

\begin{figure}[htbp]
\centering

\begin{minipage}{0.46\textwidth}
    \centering
    \includegraphics[width=\linewidth]{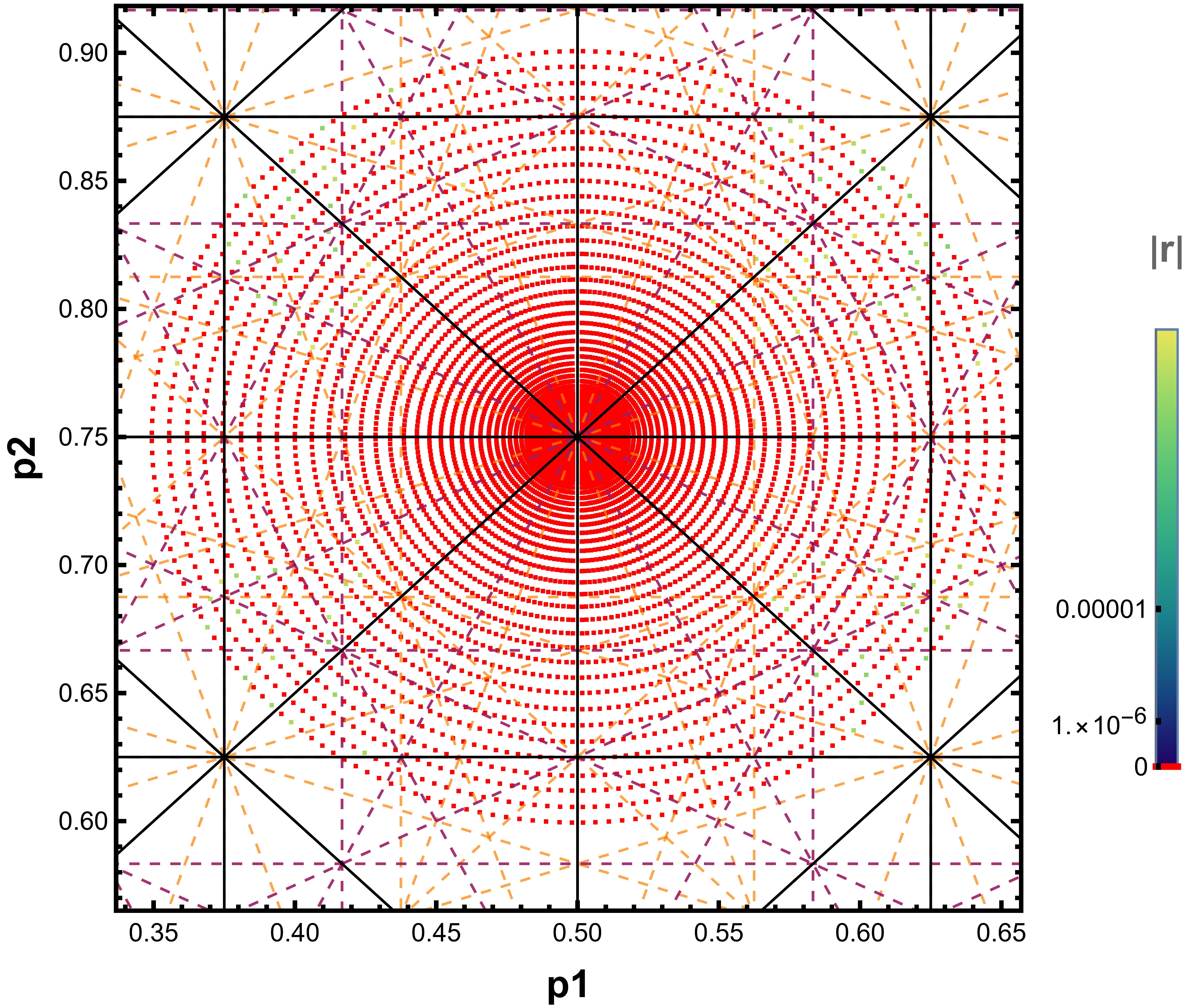}
    \par\footnotesize (a)
\end{minipage}
\hfill
\begin{minipage}{0.46\textwidth}
    \centering
    \includegraphics[width=\linewidth]{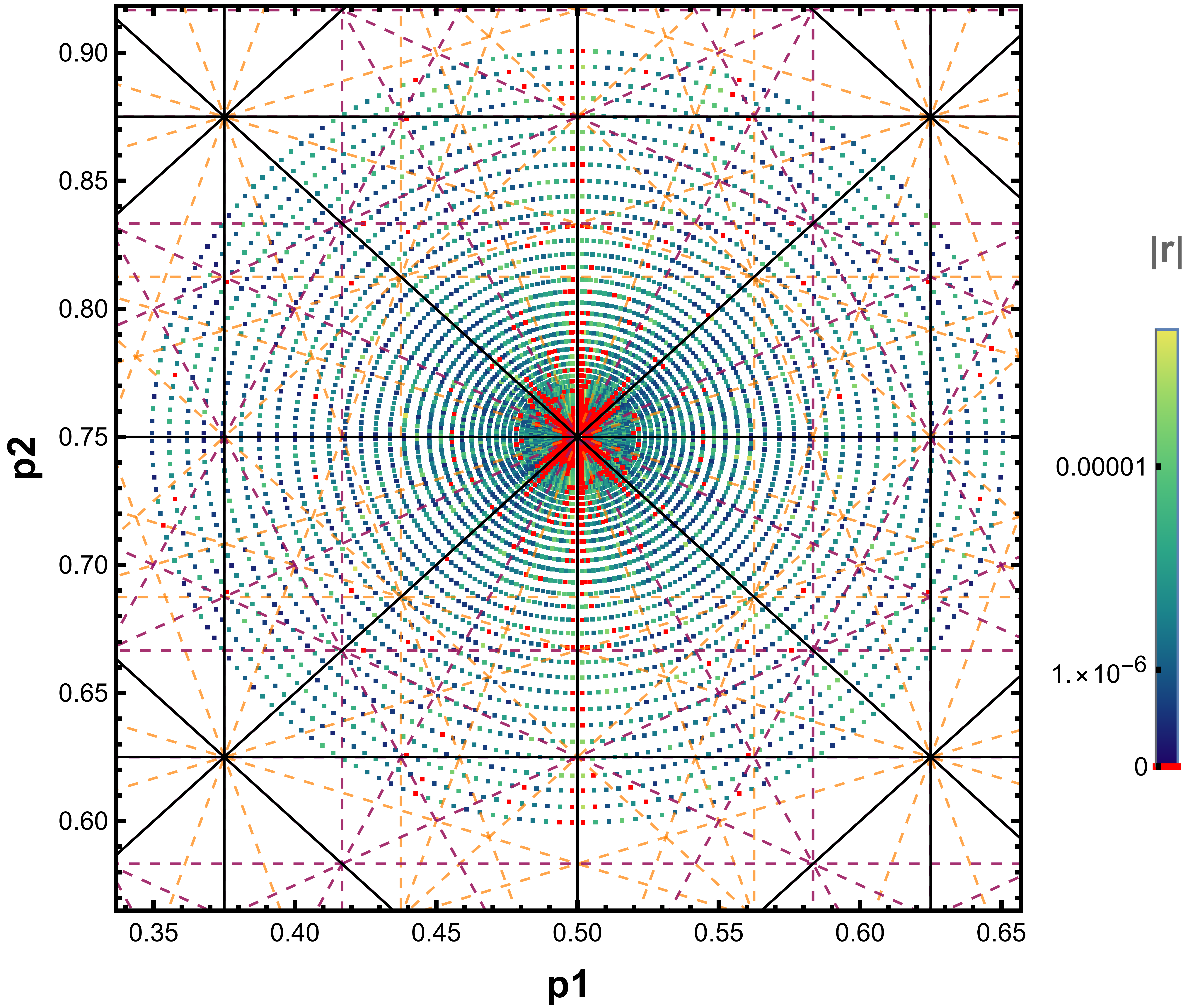}
    \par\footnotesize (b)
\end{minipage}

\vspace{1ex}

\begin{minipage}{0.46\textwidth}
    \centering
    
    \includegraphics[width=\linewidth]{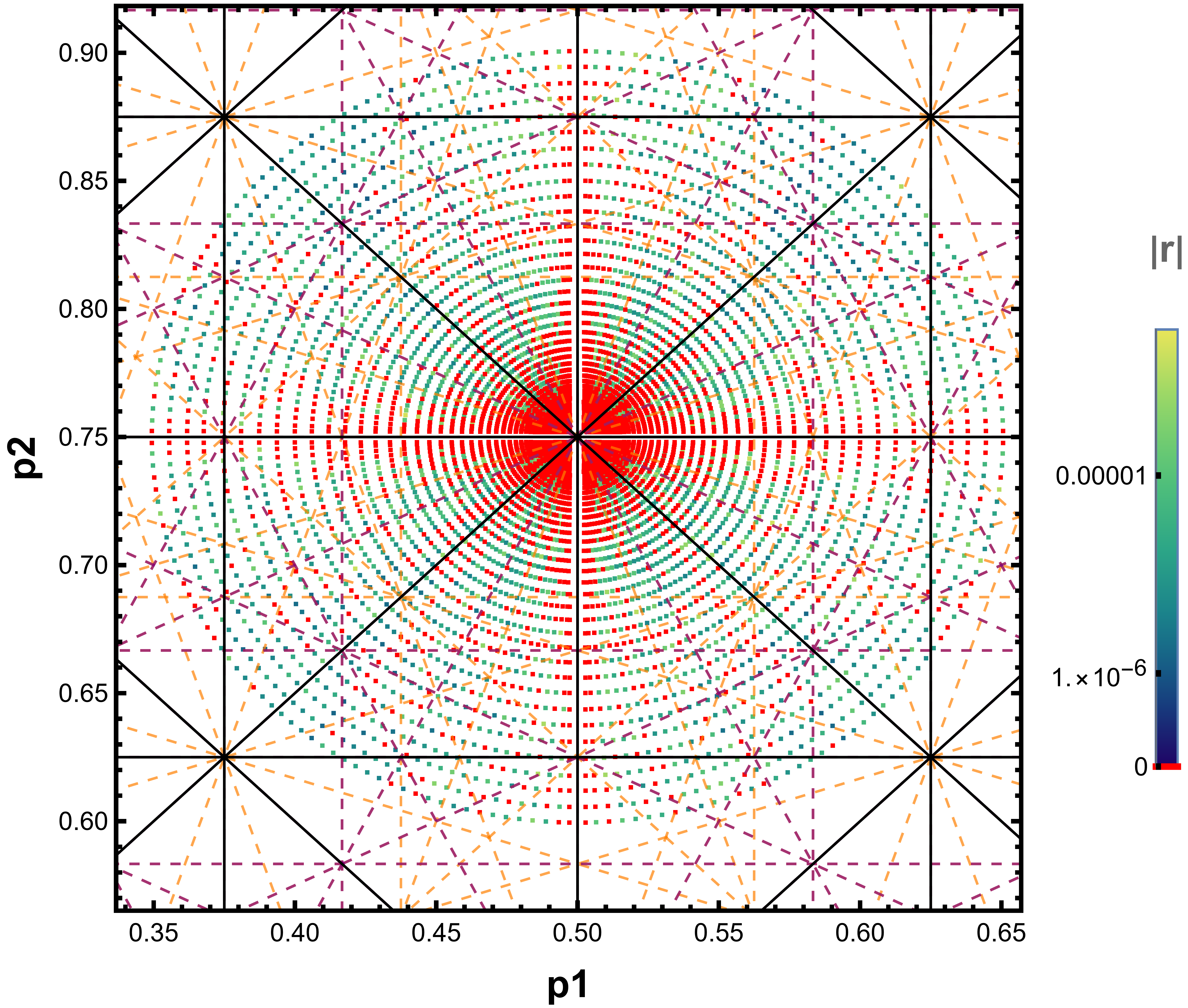}
    \par\footnotesize (c)
\end{minipage}
\hfill
\begin{minipage}{0.46\textwidth}
    \centering
    
    \includegraphics[width=\linewidth]{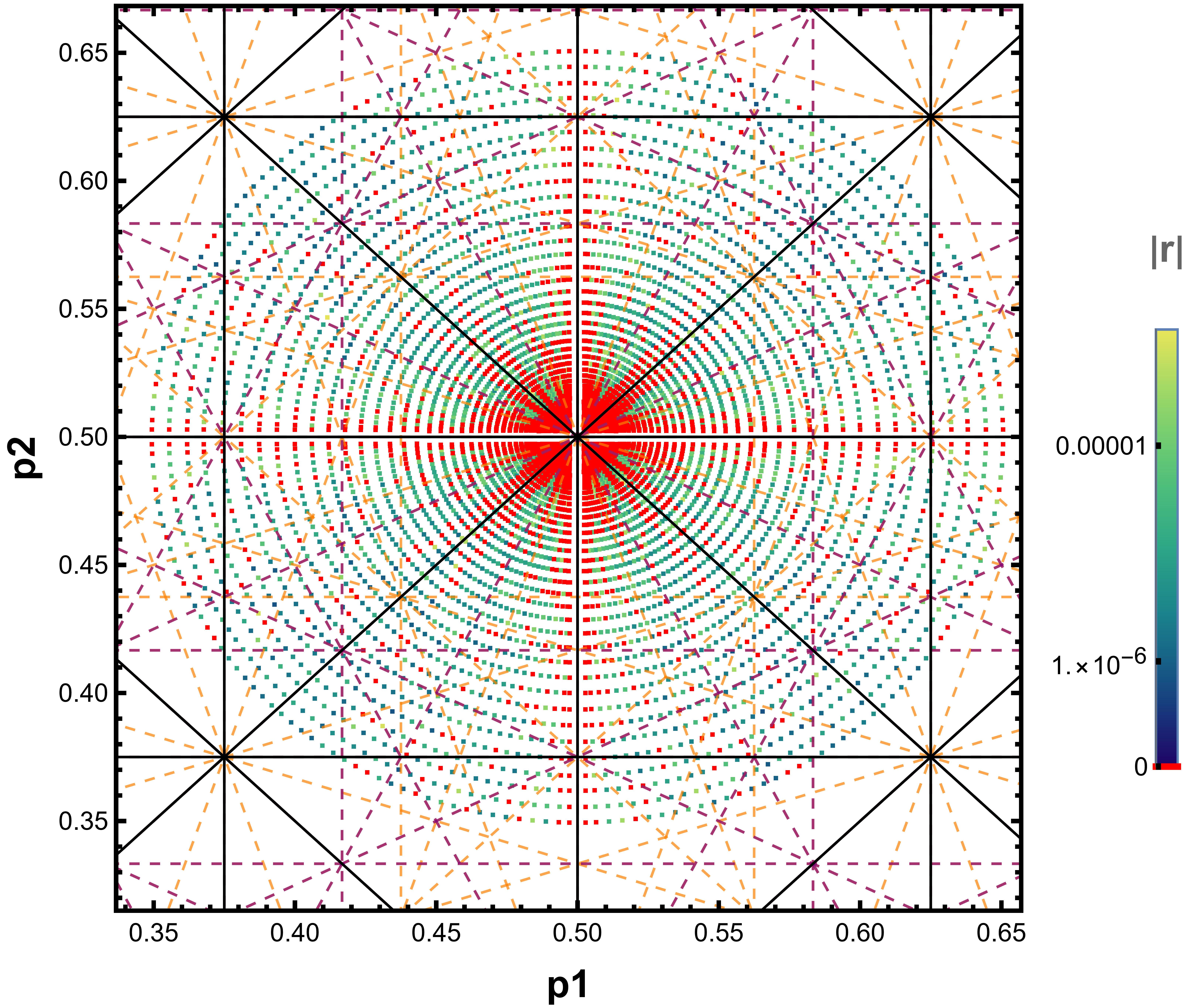}
    \par\footnotesize (d)
\end{minipage}
\vspace{-0.3cm}
\caption{\small Bound on the actions around the spin-spin-orbit resonance $(1:1)_{S_1},(3:2)_{S_2}$, for $\varepsilon_{1}=\varepsilon_{2}=10^{-5}$ after: (a) one perturbative step, (b) two perturbative steps.
Bound on the actions around the spin-spin-orbit resonance: (c) $(1:1)_{S_1},(3:2)_{S_2}$, for $\varepsilon_{1}=3\cdot 10^{-5}$ and  $\varepsilon_{2}=10^{-4}$ after two perturbative steps, (d) $(1:1)_{S_1},(1:1)_{S_2}$, for $\varepsilon_{1}=\varepsilon_{2}=3\cdot10^{-5}$ after two perturbative steps.}
\label{fig55}
\end{figure}

In Figures ~\ref{fig44} and ~\ref{fig55}, the initial conditions $(p_{1}(0),p_{2}(0))_{z}$ are represented with colored dots which form a roughly filled circle in the action plane $(p_1,p_2)$. The color of each dot indicates the value of the stability time $T$ (Figure ~\ref{fig44}) or the action bound $r$ (Figure ~\ref{fig55}) computed by implementing the algorithm.
The red points show initial conditions for which the algorithm does not provide any result (failure points). Within the action plane, the main resonances appear as continuous black lines, whereas secondary resonances are marked by dashed lines, purple for first-order resonances and orange for second-order resonances (we do not show {higher-order} resonances for illustration purposes). We stress that the initial conditions, on the $(p_1,p_2)$ plane, do not lie on the main and secondary resonances, by construction. 
In general, for initial conditions sufficiently close to resonances with large widths, it is more likely that the algorithm will not provide results. For secondary resonances of higher orders, enforced by the $\alpha$ $K$ non-resonant condition,  failure points may not appear in their proximity because of their small widths. However, they can appear near the intersections of such resonances.

Figures ~\ref{fig44}a and \ref{fig44}b show the stability times in the case of the spin-orbit/spin-orbit resonance  $(1:1)_{S_1}, (3:2)_{S_2}$; Figure ~\ref{fig44}a is obtained applying the algorithm to the Hamiltonian after one perturbative step, while Figure ~\ref{fig44}b is obtained after two perturbative steps. 
We notice that the stability times improve significantly performing one more perturbative step. 
Indeed, in Figure~\ref{fig44}a, the algorithm does not provide results for most initial conditions, while applying two perturbative steps (Figure~\ref{fig44}b) results are obtained for most of them. The points of Figure ~\ref{fig44}b for which the algorithm does not provide results are mainly concentrated along the spin-orbit resonances, particularly around the $(1:1)_{S_1}$ with the largest width, and around its intersection with the other resonances. 

Figure ~\ref{fig44}c presents results for the same spin-orbit/spin-orbit resonance after two perturbative steps, considering $\varepsilon_{2}$ larger than $\varepsilon_1$. Most of the failure points are concentrated at the intersection between the selected spin-orbit resonances, which corresponds also to an intersection of other main and secondary resonances of different orders.

Figure \ref{fig44}b and~\ref{fig44}c show how main and secondary resonances affect the stability times. In some cases, they strongly influence the stability with several failure points placed along the resonances; while in other cases, some resonances become negligible (since they have small widths) after one or two perturbative steps.

In Figure ~\ref{fig44}b, where $\varepsilon_1=\varepsilon_2$ the $(1:1)_{S_1}$ spin-orbit resonance is wider than the $(3:2)_{S_2}$, while in Figure ~\ref{fig44}c, where $\varepsilon_2>\varepsilon_1$ the domain of $(3:2)_{S_2}$ is comparable or larger than $(1:1)_{S_1}$. In Figure ~\ref{fig44}d, we consider the spin-orbit/spin-orbit resonance $(1:1)_{S_1},(1:1)_{S_2}$ with $\varepsilon_1=\varepsilon_2=3\cdot 10^{-5}$. Most of the points where the algorithm does not provide results are concentrated along both resonances and their intersection.

Regarding the bounds on the actions (Figure ~\ref{fig55}), the value of the radius $r$ is given on a color scale{,} as for the time $T$ in Figure ~\ref{fig44}. The behaviors are very close to Figure ~\ref{fig44} and therefore similar conclusions apply to Figure ~\ref{fig55}.

\section{Conclusion}\label{sec:conclusion}

The stability of resonances is of paramount importance in many physical problems within Celestial Mechanics, as resonances play a special role in the dynamics of many celestial bodies. In this work, we propose a procedure based on analytical results to analyze the stability in the vicinity of resonances. Our results are obtained by developing effective estimates based on the formulation of Nekhoroshev's theorem given in \cite{Poschel93} for the non-resonant case. Indeed, to avoid main and secondary commensurabilities, we constructed sequences of Diophantine frequencies that tend asymptotically to a fixed resonance or a selected intersection of resonances. Having in mind the applications to model problems of Celestial Mechanics, we employed distinct approaches for constructing the sequences in the one-dimensional and two-dimensional cases.

Exploiting the freedom in the choice of parameters given in \cite{Poschel93}, we developed an optimization algorithm to maximize the stability time. The estimates were further improved by reducing the norm of the perturbing function through a suitable implementation of perturbation theory, followed by the selection of the optimal parameters. 

We considered two model problems in rotational dynamics, the spin-orbit and the spin-spin-orbit problems. We studied the stability in the vicinity of the main spin-orbit and spin-spin-orbit resonances through the computation of non-resonant effective stability estimates for a grid of points constructed along the Diophantine sequences. Our results demonstrate the effectiveness of the procedure in terms of estimates of the stability times and bounds on the variation of the action variables. 

A limitation of the work presented here is the following. 
The spin-orbit and the spin-spin-orbit problems are ruled by perturbing parameters which measure the oblateness of the interacting bodies. In our samples, we have chosen values of the perturbing parameters which are small, though consistent with the astronomical data. However, one could be interested to values of the perturbing parameters, which are significantly larger than those discussed in this paper. We suggest that this analysis would require a different optimization technique as well as the development of alternative methods to deal with highly non-integrable systems. 

\section{Acknowledgements}\label{sec:acknowledgements}
\phantom{.}\vspace{-0.5cm}\\
This research has been conducted during and with
the support of the Italian National Inter-University PhD Programme in Space Science and Technology.\\
Alessia Francesca Guido acknowledges the support of INdAM group G.N.F.M.\\
Anargyros Dogkas acknowledges the support from the Italian Space Agency through the project “Monitoring Asteroids” (grant 2022-33-HH.0).\\

\appendix

\clearpage

\section{Optimization Algorithm}\label{AppendixAlgorithm}

Below we give a sketch of the optimization algorithm (see Section~\ref{sec:effective}) with a short explanation in square brackets.

\phantom{.}\\
\textit{\textbf{Input:}} initial conditions $(\underline{p}_0,\underline{q}_0) \in \mathbb{R}^n\times\mathbb{T}^n$, $r_0$\\
\textit{\textbf{Output:}} $\{T,r\}$\\

\begin{itemize}
    \item[$\blacktriangleright$] \textcolor{gray}{\textit{\big[Initialization of the parameters\big]}}\\
    ~\\
    $\omega  \varleftarrow \nabla_{\underline{p}} h (\underline{p}_{0})$,\hspace{0.3cm}$K_{min} \varleftarrow \max_{\underline{k}\in \K}(\norm{k}_1)$\\
    ~\\
    $r_S \varleftarrow \min\left(\norm{\u{p}_0-x}, \ \ x\in S=\textrm{the set of singularities of } \H(\u{p},\u{q})\right)$\\
    \textcolor{gray}{\textit{\big[Defines the analyticity radius.\big]}}\\
    ~\\
    $r_{d} \varleftarrow \min(r_{S},r_{0})$\\
    ~\\
    $M\varleftarrow \sup_{\underline{p}\in V_{r_{d}}(\underline{p}_{0})}\left(\norm{\partial^2_{\u{p}}h(\u{p})}_{2}\right)$\\

    \item[$\blacktriangleright$] \textcolor{gray}{\textit{\big[Optimization of $\ell$ and $s_0$ for given $K$\big]}}\\
    $\textrm{\textbf{Start Function}}\left(T^{max}(\ell,s_0)\right)$\\
     ~\\
     \phantom{.}\hspace{1cm} $\alpha  \varleftarrow \min_{\underline{k}\in \mathbb{Z}^{n}_{K}}|\underline{\omega}\cdot \underline{k}| $\\
    \phantom{.}\hspace{1cm}\textcolor{gray}{\textit{\big[To define the $\alpha$-$K$ non resonant domain]}}\\
    ~\\
    \phantom{.}\hspace{1cm}\textbf{Maximize} $\displaystyle \E={\frac{1}{2^7\ell}}\ \frac{\alpha}{K}\min\left(r_d,\frac{\alpha}{MK}\left(1-\frac{1}{\ell}\right)\right)$\\
    \phantom{.}\hspace{1cm}\textcolor{gray}{\textit{\big[To determine $\ell\varleftarrow \ell_{optimal}$\big]}}\\
    ~\\
    \phantom{.}\hspace{1cm}$\displaystyle\E\varleftarrow {\frac{1}{2^7\ell_{optimal}}}\ \frac{\alpha}{K}\min\left(r_d,\frac{\alpha}{MK}\left(1-\frac{1}{\ell_{optimal}}\right)\right)$\\
    ~\\
    \phantom{.}\hspace{1cm}$E(s_0)\varleftarrow \norm{\R(\u{p}_0,\u{q}_0)}_{r_0,s_0}$\\
    ~\\
    \phantom{.}\hspace{1cm}\textbf{Find Root} $E(s_0)=\E$\\
    \phantom{.}\hspace{1cm}\textcolor{gray}{\textit{\big[To determine $s_{max}\varleftarrow$ the maximum acceptable value of $s_0$\big]}}\\
    ~\\
     \phantom{.}\hspace{1cm}\textbf{Maximize} $\displaystyle T= {{s_0r}\over {5E(s_0)}}\ \exp\left(\frac{K s_0}{6}\right)$ for $s_0\in(0,s_{max}]$\\
     \phantom{.}\hspace{1cm}\textcolor{gray}{\textit{\big[To determine $s_0\varleftarrow s_{optimal}$\big]}}\\
     ~\\
     \phantom{.}\hspace{1cm}\textbf{Return}$(T(\ell_{optimal},s_{optimal}))$\\
     \textbf{End Function}\\

    \item[$\blacktriangleright$] \textcolor{gray}{\textit{\big[Optimization of $K$\big]}}\\
    $T_{max}\varleftarrow 0$,\hspace{0.3cm} $T^{(1)}_{control}\varleftarrow 0$,\hspace{0.3cm} $T^{(2)}_{control}\varleftarrow 0$\\
    $K\varleftarrow K_{min}$,\hspace{0.3cm} $dK\varleftarrow dK_{initial}$\\
    ~\\
    \textbf{While}\big($dK>dK_{control}$ and $K<K_{control}$ and $N<N_{control}$ and $T_{max}<T_{control}$\big)\\
    \phantom{.}\hspace{1cm}$K\varleftarrow K + dK$\\
    \phantom{.}\hspace{1cm}$T\varleftarrow \textrm{\textbf{Function}}\left(T^{max}(\ell,s_0)\right)$\\
    ~\\
    \phantom{.}\hspace{1cm}\textbf{If}\big($T>T_{max}$\big) $T_{max}\varleftarrow T$\hspace{0.5cm}\textbf{End If}\\
    ~\\
    \phantom{.}\hspace{1cm}\textbf{If}\big($T>T^{(1)}_{control}$\big)\\
    \phantom{.}\hspace{1cm}\phantom{.}\hspace{1cm}$T^{(2)}_{control}\varleftarrow T^{(1)}_{control}$\\\phantom{.}\hspace{1cm}\phantom{.}\hspace{1cm}$T^{(1)}_{control}\varleftarrow T$\\
    \phantom{.}\hspace{1cm}\textbf{Else}\\
    \phantom{.}\hspace{1cm}\phantom{.}\hspace{1cm} $T^{(1)}_{control}\varleftarrow T^{(2)}_{control}$\\
     \phantom{.}\hspace{1cm}\phantom{.}\hspace{1cm} $K\varleftarrow \max\left(K-2\ dK,K_{min}\right)$\\
     \phantom{.}\hspace{1cm}\phantom{.}\hspace{1cm}\textcolor{gray}{\textit{\big[To ensure that we never drop below $K_{min}$\big]}}\\
     \phantom{.}\hspace{1cm}\phantom{.}\hspace{1cm} $dK\varleftarrow c_{control}\ dK$\\
     \phantom{.}\hspace{1cm}\textbf{End If}\\
     \phantom{.}\hspace{1cm} $N\varleftarrow N+1$\\
     \textbf{End While}
     
\end{itemize}

In the above algorithm, the functions \textbf{Maximize} and \textbf{Find Root} are referring to the automatic numerical processes supported by the Wolfram Mathematica system. It is of course possible to use a custom function, like a Newton-Raphson method, but in our numerical tests the automatic functions were consistently faster and with higher precision.

For our purposes it was sufficient to set $dK_{control}=0.1$, $K_{control}=10^{5}$ (for the spin-orbit case) and $K_{control}=150$ (for the spin-spin-orbit case), $N_{control}=20$, $dK_{initial}=10$, and $c_{control}=1/3$. For the one dimensional case (the spin-orbit model) $T_{control}$ was set to infinity, while for the two dimensional case (the spin-spin-orbit model), we have used $T_{control}=10^{100}$.

Finally, some further safety conditions are applied to ensure that the result is not subject to numerical errors. To this end, we set an acceptable tolerance $tol=10^{-8}$ and we require that $\alpha>tol$. We multiply the analyticity radius $r_S$ by a safety factor, which in our examples is chosen to be $0.99$.
Lastly, initial conditions that are too close to the resonances, meaning that their back-transformation is not a near-identity transformation, or the computation of their initial actions from the relation \equ{omegaspispin} is subject to non convergence issues, are disregarded as resonant.

In Figure ~\ref{fig:Conv}{,} we consider the spin-spin-orbit model and we show the convergence of the optimization process for an initial condition that is further (gray line) or closer (orange line) to an intersection of resonances. In this example, the algorithm is converging to a good approximation of the optimal result in about 15 to 20 steps. 

\begin{figure}[h!]
    \centering
    \includegraphics[width=\linewidth]{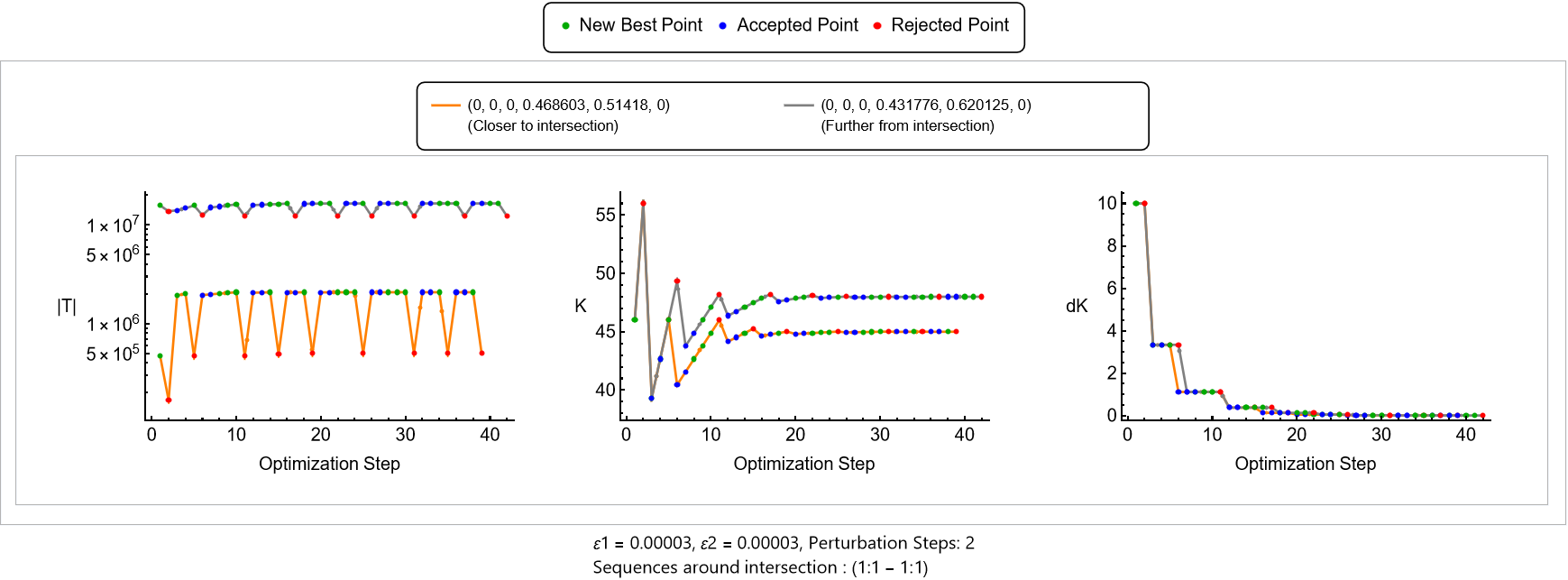}
    \caption{The values of $T$ (left), $K$ (middle), and $dK$ (right) for every step of the optimization process. We indicate with green points the optimization steps that found a new largest value of $T$, with blue the accepted in-between steps, which satisfy the conditions but do not find a new largest $T$, and with red the rejected steps, which do not satisfy at least one of the conditions of Theorem~\ref{thm:nonres}. We display two initial conditions, one closer to an intersection of resonances (orange line) and one further away (gray line).}
    \label{fig:Conv}
\end{figure}

\clearpage

\section{Expansions of the Hamiltonians of the spin-orbit and spin-spin-orbit models}\label{app:spin}

\subsection{Expansion of the spin-orbit Hamiltonian}\label{app:SO}
We expand the Hamiltonian \equ{SpiOrbitHam} in power series of the eccentricity up to degree 5. The resulting Hamiltonian is a function of the mean anomaly $q_1$ and time $q_3$; we denote by $p_1$ the momentum {conjugate} to $q_1$, and we introduce a dummy action $p_3$ {conjugate} to $q_3$. The resulting expression is the following:
\beq{VSOapp}
\begin{array}{ll}
\mathcal{H}(p_1,p_3,q_1,q_3)=&\displaystyle\frac{p_1^2}{2}+p_3 + \varepsilon  \Biggl\{
\left(-\frac{13e^4}{32}+\frac{5 e^2}{4}-\frac{1}{2}\right)\cos(2q_1-2q_3)\\~\\~&\displaystyle +\left(\frac{115e^4}{12}-\frac{17 e^2}{4}\right) \cos(2q_1-4 q_3)+\left(-\frac{11e^5}{1536}-\frac{e^3}{96}\right)\cos(2q_1+q_3)\\~\\~&\displaystyle +\left(\frac{5e^5}{768}-\frac{e^3}{32}+\frac{e}{4}\right)\cos (2 q_1-q_3)
\\~\\~&\displaystyle+\left(-\frac{489
   e^5}{256}+\frac{123 e^3}{32}-\frac{7
   e}{4}\right) \cos (2 q_1-3
   q_3)\\~\\~&\displaystyle +\left(\frac{32525
   e^5}{1536}-\frac{845 e^3}{96}\right) \cos (2
   q_1-5 q_3)-\frac{1}{48}
   e^4 \cos (2 q_1+2
   q_3)\\~\\~&\displaystyle -\frac{533}{32}
   e^4 \cos (2 q_1-6 q_3)-\frac{81 e^5}{2560}\cos
   (2 q_1+3 q_3)
\\~\\~&\displaystyle -\frac{228347 e^5 }{7680}\cos (2
   q_1-7 q_3)\Biggl\}\ . 
\end{array}
\eeq

\newpage

\subsection{Expansion of the spin-spin-orbit Hamiltonian}\label{app:SS}
Similarly to the spin-orbit Hamiltonian, we give below the expansion in power series of the eccentricity for the spin-spin-orbit Hamiltonian \equ{SpinSpinHam}:
\beqa{VSOapp}
\mathcal{H}(\u{p},\u{q})&=&\frac{{p_1}^2}{2 {I_{c}^{(1)}}}+\frac{{p_2}^2}{2 \
{I_{c}^{(2)}}}+{p_3}+m \left(-\frac{5 {I_{c}^{(1)}}^2 \
{\varepsilon_1}^2}{112 {M_{S_1}}^2}-\frac{5 {I_{c}^{(2)}}^2 \
{\varepsilon_2}^2}{112 {M_{S_2}}^2}\right)\nonumber\\
&+&\frac{m}{a^3}\sum_{k=1,2}\frac{I_{c}^{(k)}\varepsilon_k}{M_{S_k}}\Bigg(\
 -\frac{228347 e^5}{7680}\cos (2 {q_k}-7 {q_3})-\frac{81 e^5}{2560}\cos (2 {q_k}+3 {q_3})\nonumber\\
&-&\frac{533 e^4}{32} \cos (2{q_k}-6 {q_3}) -\frac{e^4}{48} \cos (2 {q_k}+2 {q_3})\nonumber\\
&+&\left(\frac{32525 e^5}{1536}-\frac{845 e^3}{96}\right) \cos (2 {q_k}-5 {q_3})+\left(\frac{115 e^4}{12}-\frac{17 \
e^2}{4}\right) \cos (2 {q_k}-4 {q_3})\nonumber\\
&+&\left(-\frac{489 \
e^5}{256}+\frac{123 e^3}{32}-\frac{7 e}{4}\right) \cos (2 {q_k}-3 {q_3})\nonumber\\
&+&\left(-\frac{13 e^4}{32}+\frac{5 e^2}{4}-\frac{1}{2}\right) \cos (2 {q_k}-2 {q_3})\nonumber\\
&+&\left(\frac{5 e^5}{768}-\frac{e^3}{32}+\frac{e}{4}\right) \cos (2 {q_k}-{q_3})+\left(-\frac{11 e^5}{1536}-\frac{e^3}{96}\right) \cos (2 \
{q_k}+{q_3})\Bigg)\nonumber\\
&+&\frac{m}{a^5}\sum_{k=1,2}\left(\frac{I_{c}^{(k)}\varepsilon_k}{M_{S_k}}\right)^2\Bigg(
 -\frac{1045005 e^5}{4096}\cos (4 {q_k}-9 {q_3})-\frac{5 e^5}{36864}\cos (4 {q_k}+{q_3})\nonumber\\
&-&\frac{123575 e^4}{1152}\cos (4 {q_k}-8 {q_3})
+\left(\frac{8822975 e^5}{36864}-\frac{93775 e^3}{2304}\right) \cos (4 {q_k}-7 {q_3})\nonumber\\
&+&\left(\frac{2675 e^4}{32}-\frac{425 \
e^2}{32}\right) \cos (4 {q_k}-6 {q_3})\nonumber\\
&+&\left(-\frac{948125 \
e^5}{18432}+\frac{6375 e^3}{256}-\frac{325 e}{96}\right) \cos (4 {q_k}-5 {q_3})\nonumber\\
&+&\left(-\frac{4975 e^4}{384}+\frac{275 \
e^2}{48}-\frac{25}{48}\right) \cos (4 {q_k}-4 {q_3})\nonumber\\
&+&\left(\frac{3275 e^5}{2048}-\frac{625 e^3}{256}+\frac{25 e}{32}\right) \cos (4 {q_k}-3 {q_3})\nonumber\\
&+&\left(\frac{25 e^4}{144}-\frac{25 e^2}{96}\right) \cos (4 {q_k}-2 {q_3})+\left(\frac{175 e^5}{36864}+\frac{25 e^3}{2304}\right) \
\cos (4 {q_k}-{q_3})\Bigg)\nonumber\\
&+&\frac{m}{a^5}\frac{I_{c}^{(1)}\varepsilon_1}{M_{S_1}} \frac{I_{c}^{(2)}\varepsilon_2}{M_{S_2}}\Bigg(-\frac{1463007 e^5}{2048}\cos(2 {q_1}+2 {q_2}-9 {q_3})
\eeqa

\beqano
&-&\frac{7
e^5}{18432}\cos (2 {q_1}+2 {q_2}+{q_3})\nonumber\\
&-&\frac{19669 e^5}{6144}\cos (2 {q_1}-2 {q_2}-5 {q_3})-\frac{19669 e^5}{6144}\cos (2 {q_1}-2 {q_2}+5 {q_3})\nonumber\\
&-&\frac{173005 e^4}{576} \cos (2 {q_1}+2 {q_2}-8 {q_3}) +\left(-\frac{105 e^4}{64}-\frac{5 e^2}{8}-\frac{1}{8}\right) \cos \
(2 {q_1}-2 {q_2})\nonumber\\
&-&\frac{745 e^4}{384} \cos (2 {q_1}-2 {q_2}-4 {q_3}) 
-\frac{745 e^4}{384} \cos (2 {q_1}-2 {q_2}+4 {q_3})\nonumber\\
&+&\left(\frac{12352165 e^5}{18432}-\frac{131285 e^3}{1152}\right) \cos (2 {q_1}+2 {q_2}-7 {q_3})\nonumber\\
&+&\left(\frac{3745 \
e^4}{16}-\frac{595 e^2}{16}\right) \cos (2 {q_1}+2 {q_2}-6 \
{q_3})\nonumber\\
&+&\left(-\frac{1327375 e^5}{9216}+\frac{8925 e^3}{128}-\frac{455 \
e}{48}\right) \cos (2 {q_1}+2 {q_2}-5 {q_3})\nonumber\\
&+&\left(-\frac{6965
e^4}{192}+\frac{385 e^2}{24}-\frac{35}{24}\right) \cos (2 {q_1}+2 {q_2}-4 {q_3})\nonumber\\
&+&\left(-\frac{4715 e^5}{2048}-\frac{145 e^3}{128}\right) \cos (2 {q_1}-2 {q_2}-3 {q_3})\nonumber\\
&+&\left(-\frac{4715 e^5}{2048}-\frac{145 e^3}{128}\right) \cos (2 {q_1}-2 {q_2}+3 {q_3})\nonumber\\
&+&\left(\frac{4585
e^5}{1024}-\frac{875 e^3}{128}+\frac{35 e}{16}\right) \cos (2 {q_1}+2 {q_2}-3 {q_3})\nonumber\\
&+&\left(-\frac{155 e^4}{96}-\frac{5 e^2}{8}\right) \cos (2 {q_1}-2 {q_2}-2 {q_3})
+\left(-\frac{155 e^4}{96}-\frac{5 e^2}{8}\right) \cos (2
{q_1}-2 {q_2}+2 {q_3})\nonumber\\
&+&\left(-\frac{7285
e^5}{3072}-\frac{135 e^3}{128}-\frac{5 e}{16}\right) \cos (2 {q_1}-2 \
{q_2}-{q_3})\nonumber\\
&+&\left(-\frac{7285 e^5}{3072}-\frac{135 \
e^3}{128}-\frac{5 e}{16}\right) \cos (2 {q_1}-2 \
{q_2}+{q_3})\nonumber\\
&+&\left(\frac{245 e^5}{18432}+\frac{35 e^3}{1152}\right) \
\cos (2 {q_1}+2 {q_2}-{q_3})\nonumber\\
&+&\left(\frac{35 e^4}{72}-\frac{35 e^2}{48}\right) \cos (2 {q_1}+2 {q_2}-2 {q_3})\Bigg)\ . 
\eeqano

\section{Implementation of perturbation theory}\label{app:perturbstep}

We consider the Hamiltonian (\ref{GeneralHamiltonian}), and following Section \ref{sec:PertTheory} we divide the Hamiltonian into: $\H(\u p,\u q)=\N(\u p)+\Z (\u p)+\R_{\underline{k}\in\Lambda}(\u{p},\u{q})+\R_{\underline{k}\in\zed_{K}^n/\Lambda}(\u{p},\u{q})$.
We apply  a Lie canonical transformation defined by 
a generating function $\chi=\chi(\underline{p}^{(1)},\underline{q}^{(1)})$.
The new Hamiltonian is obtained through the following sequence of computations:  
\beqano 
\H^{(1)}(\underline{p}^{(1)},\underline{q}^{(1)})&=&\sum_{j=0}^\infty {1\over {j!}} \L_{\chi(\underline{p}^{(1)},\underline{q}^{(1)})}^j \H(\underline{p}^{(1)},\underline{q}^{(1)})\nonumber\\
&=&\sum_{j=0}^\infty {1\over {j!}} \left(\L_{\chi(\underline{p}^{(1)},\underline{q}^{(1)})}^j \N(\underline{p}^{(1)})+\L_{\chi(\underline{p}^{(1)},\underline{q}^{(1)})}^j \Z(\underline{p}^{(1)})+\L_{\chi(\underline{p}^{(1)},\underline{q}^{(1)})}^j \R(\underline{p}^{(1)},\underline{q}^{(1)})\right)\ ; 
\eeqano 
using \equ{NFeq}, we obtain: 
\beqano 
\H^{(1)}(\underline{p}^{(1)},\underline{q}^{(1)})&=&\N(\underline{p}^{(1)})+\sum_{j=1}^\infty {1\over {j!}} \Bigg(-\L_{\chi(\underline{p}^{(1)},\underline{q}^{(1)})}^{j-1} \R_{\underline{k}\in\zed_{K}^n/\Lambda}(\underline{p}^{(1)},\underline{q}^{(1)})\Bigg)+\sum_{j=0}^\infty {1\over {j!}} \Bigg(\L_{\chi(\underline{p}^{(1)},\underline{q}^{(1)})}^j \Z(\underline{p}^{(1)})\nonumber\\
&&\hspace{1.5cm}+\L_{\chi(\underline{p}^{(1)},\underline{q}^{(1)})}^j \R_{\underline{k}\in\zed_{K}^n/\Lambda}(\underline{p}^{(1)},\underline{q}^{(1)})+\L_{\chi(\underline{p}^{(1)},\underline{q}^{(1)})}^j\R_{\underline{k}\in\Lambda}(\underline{p}^{(1)},\underline{q}^{(1)})\Bigg)\nonumber\\
&=&\N(\underline{p}^{(1)})+\sum_{j=0}^\infty \Bigg(\left(-{1\over {(j+1)!}}+{1\over {j!}}\right) \L_{\chi(\underline{p}^{(1)},\underline{q}^{(1)})}^j \R_{\underline{k}\in\zed_{K}^n/\Lambda}(\underline{p}^{(1)},\underline{q}^{(1)})\nonumber\\&&\hspace{1.5cm}+{1\over {j!}} \left(\L_{\chi(\underline{p}^{(1)},\underline{q}^{(1)})}^j \Z(\underline{p}^{(1)})+\L_{\chi(\underline{p}^{(1)},\underline{q}^{(1)})}^j\R_{\underline{k}\in\Lambda}(\underline{p}^{(1)},\underline{q}^{(1)})\right)\Bigg)\nonumber\\
&=&\N(\underline{p}^{(1)})+\Z(\underline{p}^{(1)})+\R_{\underline{k}\in\Lambda}(\underline{p}^{(1)},\underline{q}^{(1)})\nonumber\\&&\hspace{1.5cm}+\sum_{j=1}^\infty \Bigg({j\over {(j+1)!}} \L_{\chi(\underline{p}^{(1)},\underline{q}^{(1)})}^j \R_{\underline{k}\in\zed_{K}^n/\Lambda}(\underline{p}^{(1)},\underline{q}^{(1)})\nonumber\\&&\hspace{1.5cm}+{1\over {j!}} \left(\L_{\chi(\underline{p}^{(1)},\underline{q}^{(1)})}^j \Z(\underline{p}^{(1)})+\L_{\chi(\underline{p}^{(1)},\underline{q}^{(1)})}^j\R_{\underline{k}\in\Lambda}(\underline{p}^{(1)},\underline{q}^{(1)})\right)\Bigg)\nonumber\\
&=&\N(\underline{p}^{(1)})+\Z(\underline{p}^{(1)})+\R_{\underline{k}\in\Lambda}(\underline{p}^{(1)},\underline{q}^{(1)})+\R'(\underline{p}^{(1)},\underline{q}^{(1)})\ ,
\eeqano 
with
\beqano \R'(\underline{p}^{(1)},\underline{q}^{(1)})&=&\sum_{j=1}^\infty \Bigg({j\over {(j+1)!}} \L_{\chi(\underline{p}^{(1)},\underline{q}^{(1)})}^j \R_{\underline{k}\in\zed_{K}^n/\Lambda}(\underline{p}^{(1)},\underline{q}^{(1)})\nonumber\\&&\hspace{1.5cm}+{1\over {j!}} \left(\L_{\chi(\underline{p}^{(1)},\underline{q}^{(1)})}^j \Z(\underline{p}^{(1)})+\L_{\chi(\underline{p}^{(1)},\underline{q}^{(1)})}^j\R_{\underline{k}\in\Lambda}(\underline{p}^{(1)},\underline{q}^{(1)})\right)\Bigg)\ .
\eeqano 

\section{A bound on the perturbing function}\label{app:norm}
To compute the norm of the perturbing function  $\mathcal{R}(\u{p},\u{q})$ we need to calculate the expression
$$
\norm{\R(\u{p},\u{q})}_{r_0,s_0}=\sup_{\u{p}\in V_{r_0}PD}\sum_{\u{k}\in \K}\abs{R_{\u{k}}(\u{p})}\exp\left(\norm{\u{k}}_1 s_0\right),
$$
where $\K\subseteq\zed^{n}_{K}$ is the set of Fourier indices of $\mathcal{R}(\u{p},\u{q})$. 

The numerical computation of a supremum of a large function over a complex domain is computationally expensive and therefore, whenever possible, it is better to reduce the domain of the maximization. In our model problems, we are interested in, at most, a two-dimensional action space $PD\subseteq\real^2$, which is restricted to a domain of radius $r_0$ around an initial point $\u{p}_0$. Additionally, it is possible to group the terms of the perturbing function that have the same frequency $\norm{\u{k}}_1=k_z$, reducing the computation to 
$$
\norm{\R(\u{p},\u{q})}_{r_0,s_0}=\sum_{z=1}^{N}\left(\sum_{\u{k}\in [k_z]}\sup_{\u{p}\in B_{r_0}(\u{p}_0)}\abs{R_{\u{k}}(\u{p})}\right)\exp\left(k_z s_0\right),\quad [k_z]\equiv\left\{\u{k}\in\K : \norm{\u{k}}_1=k_z\right\},
$$
with $N$ the number of $[k_z]$ sets.

Furthermore, in the Hamiltonians \equ{SpiOrbitHam} and \equ{SpinSpinHam}, the coefficients of the trigonometric terms are constant. As a result, after the application of a number of perturbation steps, the actions will only appear in the denominators of the coefficients $R_{\u{k}}$ taking the following form
\beq{Rkp}
R_{\u{k}}(\u{p})=\sum_{j=1}^{N_{\u{k}}}\prod_{\xi=1}^{M^{\u{k}}_{j}}\frac{A_{\xi,j}^{\u{k}}}{(B_{\xi,j}^{\u{k}}+\u{C}_{\xi,j}^{\u{k}}\cdot\u{p})^{D_{\xi,j}^{\u{k}}}},
\eeq
where $N_{\u{k}}$ is the number of terms in $R_{\u{k}}(\u{p})$, $M^{\u{k}}_{j}$ is the number of terms at the denominator, $A_{\xi,j}^{\u{k}},B_{\xi,j}^{\u{k}},D_{\xi,j}^{\u{k}}\in\real$, and $\u{C}_{\xi,j}^{\u{k}}\in\real^n$.

Finally, we can further reduce the search for a maximum by looking for the extrema of the function in \equ{Rkp}. More specifically, given the complex expression of $\u{p}=\u{a}+i\u{b}$, $\u{a},\u{b}\in\real^n$, the coefficients $|R_{\underline{k}}(\u{p})|$ have the form:
\beq{defRj}
\abs{R_{\underline{k}}(\u{p})}=\left|\sum_{j=1}^{N_{\u{k}}}\prod_{\xi=1}^{M^{\u{k}}_{j}}\frac{A^{\u k}_{\xi,j}}{\left(\left(B^{\u k}_{\xi,j}+\u{C}^{\u k}_{\xi,j}\cdot \u{a}\right)+i\left(\u{C}^{\u{k}}_{\xi,j}\cdot\u{b}\right)\right)^{D^{\u k}_{\xi,j}}}\right|\ , \qquad \u{k}\in\K.
\eeq
To simplify the notation, denote by $\chi^{\u{k}}_{\xi,j}(\u{b}) \equiv \u{C}^{\u{k}}_{\xi,j}\cdot\u{b}, \quad
F^{\u{k},\u{a}}_{\xi,j} \equiv B^{\u{k}}_{\xi,j} + \u{C}^{\u{k}}_{\xi,j}\cdot\u{a}$. 

Using the binomial formula for each term of the product, we obtain
\beq{newdefRj}
\abs{R_{\underline{k}}(\u{p})}=\abs{\sum_{j=1}^{N_{\u{k}}}\sum_{\u{s}=\u{0}}^{\u{D}^{\u{k}}_{j}}(i)^{\norm{\u{s}}_1}\tilde{K}_{j,\u{s}}^{\u{k},\u{a}} f^{\u{k},\u{a}}_{j,\u{s}}(\u{b})}
\eeq
with

\begin{equation*}
    \begin{array}{l}
        \displaystyle\tilde{K}_{j,\u{s}}^{\u{k},\u{a}} \equiv \prod_{s_\xi\in\u{s}}(-1)^{s_\xi}\binom{D^{\u{k}}_{\xi,j}}{s_\xi}A^{\u{k}}_{\xi,j}\left(F^{\u{k},\u{a}}_{\xi,j}\right)^{D^{\u{k}}_{\xi,j}-s_\xi}  \\
         ~\\
         \displaystyle f^{\u{k},\u{a}}_{j,\u{s}}(\u{b}) = \prod_{s_\xi\in\u{s}}f^{\u{k},\u{a}}_{j,s_\xi}\equiv \prod_{s_\xi\in\u{s}}\frac{\left(\chi^{\u{k}}_{\xi,j}(\u{b})\right)^{s_\xi}}{\left(\left(F^{\u{k},\u{a}}_{\xi,j}\right)^{2}+\left(\chi^{\u{k}}_{\xi,j}(\u{b})\right)^{2}\right)^{D^{\u{k}}_{\xi,j}}}\ .
    \end{array}
\end{equation*}
Then, we split \equ{newdefRj} into its real and imaginary parts,
$$
 \abs{R_{\underline{k}}(\u{p})}=\abs{\sum^{\u{D}^{\u{k}}_{j}}_{\norm{\u{s}}_1, \text{even}} \sum_{j=1}^{N_{\u{k}}} \tilde K^{\u{k},\u{a}}_{j,\u{s}} (-1)^{[\norm{s}_1/2]} f_{j,\u{s}}^{\u{k},\u{a}}(\u{b})+ i\sum^{\u{D}^{\u{k}}_{j}}_{\norm{\u{s}}_1, \text{odd}}\sum_{j=1}^{N_{\u{k}}}\tilde K^{\u{k},\u{a}}_{j,\u{s}}(-1)^{[\norm{\u{s}}_1/2]} f_{j,\u{s}}^{\u{k},\u{a}}(\u{b})}.
$$
Let us determine the location of the extrema on the imaginary plane, setting $\u{a}$ constant and considering the derivative over $\u{b}$. If we denote by 
$$
u(\u{b}) \equiv \textrm{Re}\left(\abs{R_{\underline{k}}(\u{p})}\right), \qquad
v(\u{b}) \equiv \textrm{Im}\left(\abs{R_{\underline{k}}(\u{p})}\right),
$$
then, the partial derivative of $R_{\underline{k}}(\u{p})$ is 
\beq{derRj}
\frac{\partial \abs{R_{\u k}(\u{p})}}{\partial b_{h}}=\frac{u(\u{b})u'(\u{b})+v(\u{b}) v'(\u{b})}{\sqrt{u^{2}(\u{b})+v^{2}(\u{b})}} \qquad\text{for} \quad h=1,2
\eeq
with 
\beqano
u'(\u{b})&\equiv&\frac{\partial u(\u{b})}{\partial b_{h}}=\sum^{\u{D}^{\u{k}}_{j}}_{\norm{\u{s}}_1, \textrm{even}}\sum_{j=1}^{N_{\u{k}}}\tilde K^{\u{k},\u{a}}_{j,\u{s}}(-1)^{[\norm{\u{s}}_1/2]} \frac{\partial f^{\u{k},\u{a}}_{j,\u{s}}(\u{b})}{\partial b_{h}} \nonumber\\ 
v'(\u{b})&\equiv&\frac{\partial v(\u{b}) }{\partial b_{h}}=\sum^{\u{D}^{\u{k}}_{j}}_{\norm{\u{s}}_1, \textrm{odd}}\sum_{j=1}^{N_{\u{k}}}\tilde K^{\u{k},\u{a}}_{j,\u{s}}(-1)^{[\norm{\u{s}}_1/2]}\frac{\partial f^{\u{k},\u{a}}_{j,\u{s}}(\u{b})}{\partial b_{h}}\ .
\eeqano
Computing the zeros of \equ{derRj} can be reduced to finding the zeros of $u'(\u{b})$ and $v'(\u{b})$, which are the critical points of $f^{\u{k},\u{a}}_{j,\u{s}}(\u{b})$. These are found by the points $\u{b}\in\real^n$ where the derivatives of $f^{\u{k},\u{a}}_{j,\u{s}}$ vanish:
\beq{equationsdfb}
\frac{\partial f^{\u{k},\u{a}}_{j,\u{s}}(\u{b})}{\partial b_h}=\sum_{\xi=1}^{M_{j}^{\u{k}}}\prod_{\substack{s_d\in \u{s}\\d\neq \xi}}f_{j,s_d}^{\u{k},\u{a}}(\u{b})\frac{\partial f_{j,s_\xi}^{\u{k},\u{a}}(\u{b})}{\partial b_h}=0, \quad \forall \quad  
\begin{array}{l}
     \u{s}=\{s_\xi\}_{\xi=1}^{M_j^{\u{k}}},  \ s_{\xi}=1,...,D^{\u{k}}_{\xi,j} \\
    ~\\
    j=1,...,N_{\u{k}}, \ \u{k}\in\K.
\end{array}
\eeq
Except for some special cases, only the trivial solution with all terms $\partial_{b_h} f_{j,s_\xi}^{\u{k},\u{a}}(\u{b})=0$ satisfy equations \equ{equationsdfb} at the same time for all set of indexes $\u s$, $j$, $\u k$ and for all values of $\u a$. These solutions are equivalent to {requiring} that 
\beq{sol1}
\left\{\begin{array}{l}
     \displaystyle\chi^{\u{k}}_{\xi,j}(\u{b})=0  \\
     ~\\
    \textrm{\textbf{or}}
     ~\\
     \displaystyle\chi^{\u{k}}_{\xi,j}(\u{b})=\pm S^{\u{k}}_{\xi,j}=\pm \sqrt{\frac{s_\xi \left(F^{\u{k},\u{a}}_{\xi,j}\right)^2}{(2D^{\u{k}}_{\xi,j}-s_\xi)}}
\end{array}\right. \quad 
\begin{array}{ll}
     \forall& \ s_\xi=1,...,D^{\u{k}}_{\xi,j}, \ \ \xi=1,..., M^{\u{k}}_{j}  \\
     ~\\
     ~& \ j=1,..,N_{\u{k}}, \ \ \u{k}\in\K.
\end{array}
\eeq
Let us assume that $\chi^{\u{k}}_{\xi,j}(\u{b})=0$ is a maximum. Considering that the perturbing function is a continuous and smooth function over the complex domain of definition, that for $\norm{\u{b}}\varrightarrow\infty$ we get $\abs{R_{k}}\varrightarrow$ constant, and that the solutions $\chi^{\u{k}}_{\xi,j}(\u{b})=\pm S^{\u{k}}_{\xi,j}$ are appearing symmetrically around $\chi^{\u{k}}_{\xi,j}(\u{b})=0$, we can claim that along any section of the domain, any solutions $\chi^{\u{k}}_{\xi,j}(\u{b})=\pm S^{\u{k}}_{\xi,j}$, if they exist, have to be minima or saddle points. This statement is based on the following argument. Consider all $\{b_d\}_{d\neq r_1}$ and $\{a_d\}_{d\neq r_2}$ to be constant, then on the plane $(a_{r_2},b_{r_1})$ consider a curve that is perpendicular to the solution $\chi^{\u{k}}_{\xi,j}(\u{b})=0$. Along this curve the function has a maximum on $\chi^{\u{k}}_{\xi,j}(\u{b})=0$ and tends to a constant value asymptotically when $b_{r_1}\varrightarrow\pm\infty$. As a consequence, the extrema $\chi^{\u{k}}_{\xi,j}(\u{b})=\pm S^{\u{k}}_{\xi,j}$, appearing symmetrically around the maximum, have to be either minima or saddle points, by the continuity and smoothness of the function $\abs{R_{k}}$. As a result, if the manifold $\chi^{\u{k}}_{\xi,j}(\u{b})=0$ defines a maximum, then it also defines the global maximum for each section corresponding to a given value of $\u{a}$.

Generally, the system of equations $\chi^{\u{k}}_{\xi,j}(\u{b})=0, \  \forall \xi=1,..., M^{\u{k}}_{j}, \ j=1,.., N_{\u{k}}, \ \ \u{k}\in\K$ defines a large number of equations for the determination of a vector $\u{b}\in\real^n$, much larger than the dimension $n$ in our sample models (in which $n$ is less or equal than 2). 
As before, with some possible special exceptions, only the trivial solutions $\chi^{\u{k}}_{\xi,j}(\u{b})=0$ for $\u{b}=\u{0}$ satisfy all equations.

Finally, we proceed to show that the extrema $\u{b}=\u{0}$ are indeed maxima. Recalling that $ u'(\u 0)=v'(\u 0)=0$, we evaluate the second partial derivatives at $\u{b}=\u{0}$ as 
\beq{secderRj}
\left.\frac{\partial^{2}|R_{\underline{k}}|}{\partial b^{2}_{h}}\right|_{\u b=\u{0}}=\left.\frac{(u(\u{b})u^{\prime\prime}(\u{b})+v(\u{b})v^{\prime\prime}(\u{b}))}{\sqrt{u^{2}(\u{b})+v^{2}(\u{b})}}\right|_{\u b=\u{0}}\ ,
\qquad  \qquad h=1,2 
\eeq
with 
\beqano
u''(\u{b})&=&\frac{\partial^2 u(\u{b})}{\partial b_{h}^2}=\sum^{\u{D}^{\u{k}}_{j}}_{\norm{\u{s}}_1, \textrm{even}}\sum_{j=1}^{N_{\u{k}}}\tilde K^{\u{k},\u{a}}_{j,\u{s}}(-1)^{[\norm{\u{s}}_1/2]} \frac{\partial^2 f^{\u{k},\u{a}}_{j,\u{s}}(\u{b})}{\partial b_{h}^2} \nonumber\\ 
v''(\u{b})&=&\frac{\partial^2 v(\u{b}) }{\partial b_{h}^2}=\sum^{\u{D}^{\u{k}}_{j}}_{\norm{\u{s}}_1, \textrm{odd}}\sum_{j=1}^{N_{\u{k}}}\tilde K^{\u{k},\u{a}}_{j,\u{s}}(-1)^{[\norm{\u{s}}_1/2]}\frac{\partial^2 f^{\u{k},\u{a}}_{j,\u{s}}(\u{b})}{\partial b_{h}^2}\ .
\eeqano
All terms with $\u{s}\neq\u{0}$ vanish, since $\chi_{{\xi},j}^{\u{k}}(\u{b})=0$. As a consequence, we have 
$$ \left.\frac{\partial^{2}\abs{R_{\underline{k}}}}{\partial b^{2}_{h}}\right|_{\u b=\u{0}}=\frac{1}{{|R_{\underline{k}}|}}\sum_{j=1}^{N_{\u{k}}}\sum_{j'=1}^{N_{\u{k}}}\tilde{K}^{\u{k},\u{a}}_{j,\u{0}}\tilde{K}^{\u{k},\u{a}}_{j',\u{0}} f^{\u{k},\u{a}}_{j,\u{0}}(\u{0})\left.\frac{\partial ^{2}f^{\u{k},\u{a}}_{j',\u{0}}(\u{b})}{\partial b^{2}_{h}}\right|_{\u b=\u{0}},
$$
whose sign is given by 
\beq{signsecderR}
\sgn\left(\left.\frac{\partial^{2}\abs{R_{\underline{k}}}}{\partial b^{2}_{h}}\right|_{\u b=\u{0}}\right) = \sgn\left(\sum_{j=1}^{N_{\u{k}}}\sum_{j'=1}^{N_{\u{k}}}\tilde{K}^{\u{k},\u{a}}_{j,\u{0}}\tilde{K}^{\u{k},\u{a}}_{j',\u{0}} f^{\u{k},\u{a}}_{j,\u{0}}(\u{0})\left.\frac{\partial ^{2}f^{\u{k},\u{a}}_{j',\u{0}}(\u{b})}{\partial b^{2}_{h}}\right|_{\u b=\u{0}}\right) 
\eeq
with 
\beq{secderf}
\left\{\begin{array}{l}
     \displaystyle\left.\frac{\partial ^{2}f^{\u{k},\u{a}}_{j',\u{0}}(\u{b})}{\partial b^{2}_{h}}\right|_{\u{b}=\u{0}} = \sum_{\xi=1}^{M_j^{\u{k}}}\prod_{\substack{d=1\\d\neq\xi}}^{M_j^{\u{k}}}\frac{1}{\left(F^{\u{k},\u{a}}_{d,j'}\right)^{2D^{\u{k}}_{d,j'}}}\frac{-2D^{\u{k}}_{\xi,j'} \left(C^{\u{k}}_{\xi,j',h}\right)^{2}}{\left(F^{\u{k},\u{a}}_{\xi,j'}\right)^{2(D^{\u{k}}_{\xi,j'}+1)}}  \\
     ~\\
     \displaystyle f^{\u{k},\u{a}}_{j',\u{0}}(\u{0}) = \prod_{\xi=1}^{M_j^{\u{k}}}\frac{1}{\left(F^{\u{k},\u{a}}_{\xi,j}\right)^{2D^{\u{k}}_{\xi,j}}}
\end{array}\right. .
\eeq
Using the relations \equ{signsecderR} and \equ{secderf}, we verify numerically that 
$$
\left.\frac{\partial^{2}\abs{R_{\underline{k}}}}{\partial b^{2}_{h}}\right|_{\u{b}=\u{0}}<0 
$$
for the Hamiltonians \equ{SpiOrbitHam} and \equ{SpinSpinHam}, as well as for the Hamiltonians resulting from the application of perturbation theory. Due to the fact that all mixed derivatives are zero (namely the off-diagonal terms of the Hessian matrix), the Hessian is negative definite at $\u{b}=\u{0}$. Therefore, as we have shown previously, $\u{b}=\u{0}$ is a global maximum for the Hamiltonians \equ{SpiOrbitHam} and \equ{SpinSpinHam}, which implies that the supremum of each $|R_{\underline{k}}(\u{p})|$ is located in the domain $\u{p}\in P_{\real^n}B_{r_0}(\underline{p}_{0})$.

 Finally, since $|R_{\underline{k}}(\underline{p})|$ is analytic in the domain $B_{r_0}(\underline{p}_0)$, by the maximum modulus principle the maximum appears on the boundary $\partial B_{r_0}(\underline{p}_0)$. However, we found that the supremum is located in $\u{p}\in P_{\real^n}B_{r_0}(\underline{p}_{0})$. Therefore, the maximum will appear on their intersection $\u{p}\in \partial (P_{\real^n}B_{r_0}(\underline{p}_{0}))$, further restricting the domain where the maximum is located.

In two dimensions, using a polar parametrization $(r,\theta)$ for $(a_1,a_2)$, we can write the conclusion as  
\begin{equation*}
\sup_{B_{r_0}(\underline{p}_{0})}|R_{\underline{k}}(a_{1},a_{2},b_{1},b_{2})|=\sup_{(a_{1},a_{2})\in \partial\left(P_{\real^2}B_{r_0}(\u p_0)\right)}|R_{\underline{k}}(a_{1},a_{2},0,0)|=\sup_{\theta\in [0,2\pi ] }|R_{\underline{k}}(\theta,r_
{0})|.
\end{equation*}
In one dimension, the search for the maximum is reduced to the comparison of just two points $a_1=p_0\pm r_0$.

\bibliographystyle{abbrv}
\bibliography{biblioAAA}

@article {Moser62,
MR = MR26:5255,
AUTHOR = {Moser, Jurgen},
TITLE = {On invariant curves of area-preserving mappings of an annulus},
JOURNAL = {Nachr. Akad. Wiss. G\"ottingen Math.-Phys. Kl. II},
VOLUME = {1962},
YEAR = {1962},
PAGES = {1--20},
}

@article {Arnold63a,
AUTHOR = {{A}rnol'd, Vladimir I.},
TITLE = {Proof of a theorem of {A}. {N}. {K}olmogorov on the invariance
of quasi-periodic motions under small perturbations},
JOURNAL = {Russian Math. Surveys},
FJOURNAL = {Russian Mathematical Surveys},
VOLUME = {18},
YEAR = {1963},
NUMBER = {5},
PAGES = {9--36},
}

@book {Celletti2010,
    AUTHOR = {Celletti, Alessandra},
     TITLE = {Stability and Chaos in Celestial Mechanics},
 PUBLISHER = {Springer-Verlag, Berlin; published in association with Praxis
              Publishing, Chichester},
      YEAR = {2010},
     PAGES = {xvi+261},
      ISBN = {978-3-540-85145-5},
   MRCLASS = {70-02 (37J40 37N05 70F05 70F15 70H14 70M20)},
  MRNUMBER = {2571993},
MRREVIEWER = {Manuele Santoprete},
       DOI = {10.1007/978-3-540-85146-2},
       URL = {http://dx.doi.org/10.1007/978-3-540-85146-2},
}

@article {Nekhoroshev77,
AUTHOR = {Nekhoroshev, Nikolay N.},
TITLE = {An exponential estimate of the time of stability of nearly
integrable {H}amiltonian systems},
NOTE = {English translation:
{\em Russian Math. Surveys}, 32(6):\-1--65, 1977},
JOURNAL = {Uspehi Mat. Nauk},
VOLUME = {32},
YEAR = {1977},
NUMBER = {6(198)},
PAGES = {5--66, 287},
MRCLASS = {58F05 (34D05 70.34)},
MRNUMBER = {58 #18570},
MRREVR = {Ju. E. Gliklih},
}

@article {Poschel93,
AUTHOR = {P{\"o}schel, Jurgen},
TITLE = {Nekhoroshev estimates for quasi-convex {H}amiltonian systems},
JOURNAL = {Math. Z.},
FJOURNAL = {Mathematische Zeitschrift},
VOLUME = {213},
YEAR = {1993},
NUMBER = {2},
PAGES = {187--216},
ISSN = {0025-5874},
CODEN = {MAZEAX},
MRCLASS = {58F05 (34C99 58F30 70H05)},
MRNUMBER = {94m:58089},
MRREVR = {A. Giorgilli},
}

@article {Kolmogorov54,
AUTHOR = {Kolmogorov, Andrej N.},
TITLE = {On conservation of conditionally periodic motions for a small
change in {H}amilton's function},
NOTE = {English translation in
{\it Stochastic Behavior in Classical and Quantum Hamiltonian
Systems (Volta Memorial Conf., Como, 1977)},
Lecture Notes in Phys., 93,
pages 51--56. Springer, Berlin, 1979},
JOURNAL = {Dokl. Akad. Nauk SSSR (N.S.)},
VOLUME = {98},
YEAR = {1954},
PAGES = {527--530},
MRCLASS = {36.0X},
MRNUMBER = {16,924c},
MRREVR = {Y. N. Dowker},
}

@article{Greene79,
AUTHOR = {J. M. Greene},
TITLE =  {A method for determining a stochastic transition},
JOURNAL = {Jour. Math. Phys.},
FJOURNAL = {Journal of Mathematical Physics},
VOLUME = {20},
PAGES = {1183--1201},
YEAR = {1979},
}

@article {Chirikov79,
AUTHOR = {Chirikov, B.V.},
TITLE = {A universal instability of many-dimensional oscillator systems},
JOURNAL = {Phys. Rep.},
FJOURNAL = {Physics Reports. Section C of Physics Letters},
VOLUME = {52},
YEAR = {1979},
NUMBER = {5},
PAGES = {264--379},
ISSN = {0370-1573},
CODEN = {PRPLCM},
MRCLASS = {70K20 (58F05 70H05)},
MRNUMBER = {80h:70022},
MRREVR = {Richard C. Churchill},
}

@preamble{
"\def\cprime{$'$} "
}

@article {Celletti90I,
    AUTHOR = {Celletti, Alessandra},
     TITLE = {Analysis of resonances in the spin-orbit problem in celestial
              mechanics: the synchronous resonance. {I}},
   JOURNAL = {Z. Angew. Math. Phys.},
  FJOURNAL = {Zeitschrift f\"ur Angewandte Mathematik und Physik. ZAMP.
              Journal of Applied Mathematics and Physics. Journal de
              Math\'ematiques et de Physique Appliqu\'ees},
    VOLUME = {41},
      YEAR = {1990},
    NUMBER = {2},
     PAGES = {174--204},
      ISSN = {0044-2275},
   MRCLASS = {70F15 (58F05 70H05 70K30)},
  MRNUMBER = {1045811},
MRREVIEWER = {James A. Murdock},
       DOI = {10.1007/BF00945107},
       URL = {http://dx.doi.org/10.1007/BF00945107},
}

@ARTICLE{Wisdom,
       author = {{Wisdom}, J. and {Peale}, S.~J. and {Mignard}, F.},
        title = "{The chaotic rotation of Hyperion}",
      journal = {Icarus},
         year = 1984,
        month = may,
       volume = {58},
       number = {2},
        pages = {137-152},
          doi = {10.1016/0019-1035(84)90032-0},
       adsurl = {https://ui.adsabs.harvard.edu/abs/1984Icar...58..137W}
}

@article {CellettiFL04,
    AUTHOR = {Celletti, A. and Falcolini, C. and Locatelli, U.},
     TITLE = {On the break-down threshold of invariant tori in four
              dimensional maps},
   JOURNAL = {Regul. Chaotic Dyn.},
  FJOURNAL = {Regular \& Chaotic Dynamics. International Scientific Journal},
    VOLUME = {9},
      YEAR = {2004},
    NUMBER = {3},
     PAGES = {227--253},
      ISSN = {1560-3547},
   MRCLASS = {37J40 (37J10 65P40 70K43 70K55)},
  MRNUMBER = {2104170},
MRREVIEWER = {Sergei A. Dovbysh},
       DOI = {10.1070/RD2004v009n03ABEH000278},
       URL = {https://doi.org/10.1070/RD2004v009n03ABEH000278},
}

@ARTICLE{CG,
       author = {{Celletti}, Alessandra and {Giorgilli}, Antonio},
        title = "{On the stability of the {L}agrangian points in the spatial restricted problem of three bodies}",
      journal = {Celestial Mechanics and Dynamical Astronomy},
         year = 1990,
        month = mar,
       volume = {50},
       number = {1},
        pages = {31-58},
          doi = {10.1007/BF00048985},
       adsurl = {https://ui.adsabs.harvard.edu/abs/1990CeMDA..50...31C}
}

@ARTICLE{GS,
       author = {{Giorgilli}, Antonio and {Skokos}, Charalampos},
        title = "{On the stability of the Trojan asteroids.}",
      journal = {Astron. $\&$ Astroph.},
         year = 1997,
        month = jan,
       volume = {317},
        pages = {254-261},
       adsurl = {https://ui.adsabs.harvard.edu/abs/1997A&A...317..254G}
}

@ARTICLE{CDE,
       author = {{Celletti}, Alessandra and {De Blasi}, Irene and {Efthymiopoulos}, Christos},
        title = "{Nekhoroshev estimates for the orbital stability of Earth's satellites}",
      journal = {Celestial Mechanics and Dynamical Astronomy},
         year = 2023,
        month = apr,
       volume = {135},
       number = {2},
          eid = {10},
        pages = {10},
          doi = {10.1007/s10569-023-10124-9},
       adsurl = {https://ui.adsabs.harvard.edu/abs/2023CeMDA.135...10C}
}

@ARTICLE{CF,
       author = {{Celletti}, Alessandra and {Ferrara}, Laura},
        title = "{An Application of the Nekhoroshev Theorem to the Restricted Three-Body Problem}",
      journal = {Celestial Mechanics and Dynamical Astronomy},
         year = 1996,
        month = sep,
       volume = {64},
       number = {3},
        pages = {261-272},
          doi = {10.1007/BF00728351},
       adsurl = {https://ui.adsabs.harvard.edu/abs/1996CeMDA..64..261C}
}

@ARTICLE{Lochak,
       author = {{Lochak}, P. and {Neishtadt}, A.~I.},
        title = "{Estimates of stability time for nearly integrable systems with a quasiconvex Hamiltonian}",
      journal = {Chaos},
         year = 1992,
        month = oct,
       volume = {2},
       number = {4},
        pages = {495-499},
          doi = {10.1063/1.165891},
       adsurl = {https://ui.adsabs.harvard.edu/abs/1992Chaos...2..495L}
}

@article{goldreich1966spin,
  title={Spin-orbit coupling in the solar system},
  author={Goldreich, Peter and Peale, Stanton},
  journal={Astronomical Journal, Vol. 71, p. 425 (1966)},
  volume={71},
  pages={425},
  year={1966}
}

@book {Khinchin,
	AUTHOR = {Khinchin, A.Ya.},
	TITLE = {Continued Fractions},
	PUBLISHER = {Dover Publications},
	ADDRESS = {Mineola, New York},
	YEAR = {1964},
}

@BOOK{Ferraz,
       author = {{Ferraz-Mello}, Sylvio},
        title = "{Canonical Perturbation Theories - Degenerate Systems and Resonance}",
         year = 2007,
       volume = {345},
          doi = {10.1007/978-0-387-38905-9},
       adsurl = {https://ui.adsabs.harvard.edu/abs/2007ASSL..345.....F}
}

@ARTICLE{Laplata,
       author = {{Efthymiopoulos}, C.},
        title = "{Canonical perturbation theory; stability and diffusion in Hamiltonian systems: applications in dynamical astronomy}",
      journal = {Workshop Series of the Asociacion Argentina de Astronomia},
         year = 2011,
        month = jan,
       volume = {3},
        pages = {3-146},
       adsurl = {https://ui.adsabs.harvard.edu/abs/2011WSAAA...3....3E}
}

@ARTICLE{CGM,
       author = {{Celletti}, Alessandra and {Gimeno}, Joan and {Misquero}, Mauricio},
        title = "{The Spin-Spin Problem in Celestial Mechanics}",
      journal = {Journal of NonLinear Science},
         year = 2022,
        month = dec,
       volume = {32},
       number = {6},
          eid = {88},
        pages = {88},
          doi = {10.1007/s00332-022-09840-7},
archivePrefix = {arXiv},
       eprint = {2110.11152},
 primaryClass = {math.DS},
       adsurl = {https://ui.adsabs.harvard.edu/abs/2022JNS....32...88C}
}

@book{Whittaker_Watson_2021, 
place={Cambridge}, 
edition={5}, 
title={A Course of Modern Analysis}, 
publisher={Cambridge University Press}, 
author={Whittaker, E. T. and Watson, G. N.}, 
editor={Moll, Victor H.Editor}, 
year={2021}}

@BOOK{Hagihara_1970_V1,
       author = {{Hagihara}, Yusuke},
        title = "{Celestial mechanics. Vol.1: Dynamical principles and transformation theory}",
         year = {1970},
       adsurl = {https://ui.adsabs.harvard.edu/abs/1970ceme.book.....H}
}

@ARTICLE{Efthymiopoulos_Giorgilli_Contopoulos_2004,
       author = {{Efthymiopoulos}, C. and {Giorgilli}, A. and {Contopoulos}, G.},
        title = "{Nonconvergence of formal integrals: II. Improved estimates for the optimal order of truncation}",
      journal = {Journal of Physics A Mathematical General},
         year = 2004,
        month = nov,
       volume = {37},
       number = {45},
        pages = {10831-10858},
          doi = {10.1088/0305-4470/37/45/008},
       adsurl = {https://ui.adsabs.harvard.edu/abs/2004JPhA...3710831E}
}

@ARTICLE{BCL25,
       author = {{Bustamante}, Adri{\'a}n P. and {Celletti}, Alessandra and {Lhotka}, Christoph},
        title = "{The dynamics of the spin{\textendash}spin problem in Celestial Mechanics}",
      journal = {Communications in Nonlinear Science and Numerical Simulations},
         year = 2025,
        month = mar,
       volume = {142},
          eid = {108548},
        pages = {108548},
          doi = {10.1016/j.cnsns.2024.108548},
       adsurl = {https://ui.adsabs.harvard.edu/abs/2025CNSNS.14208548B}
}

@ARTICLE{Misquero,
       author = {{Misquero}, Mauricio},
        title = "{The spin-spin model and the capture into the double synchronous resonance}",
      journal = {Nonlinearity},
         year = 2021,
        month = apr,
       volume = {34},
       number = {4},
        pages = {2191-2219},
          doi = {10.1088/1361-6544/abc4d8},
archivePrefix = {arXiv},
       eprint = {2010.09354},
 primaryClass = {math.DS},
       adsurl = {https://ui.adsabs.harvard.edu/abs/2021Nonli..34.2191M}
}

@article{Lochack92,
author = {P Lochak},
title = {Canonical perturbation theory via simultaneous approximation},
journal = {Russian Mathematical Surveys},
year = {1992},
month = {dec},
volume = {47},
number = {6},
pages = {57},
doi = {10.1070/RM1992v047n06ABEH000965},
url = {https://dx.doi.org/10.1070/RM1992v047n06ABEH000965},
publisher = {},

}

@article{Bounemoura2011,
author = {Bounemoura, Abed and Marco, Jean-Pierre},
title = {Improved exponential stability for near-integrable quasi-convex Hamiltonians},
journal = {Nonlinearity},
year = {2010},
month = {nov},
volume = {24},
number = {1},
pages = {97},
doi = {10.1088/0951-7715/24/1/005},
url = {https://dx.doi.org/10.1088/0951-7715/24/1/005},
publisher = {},

}

@article{Zhang2011,
author = {Zhang, Ke},
title = {Speed of Arnold diffusion for analytic Hamiltonian systems},
journal = {Inventiones Mathematicae},
year = {2011},
month = {11},
volume = {186},
pages = {255-290},
doi = {10.1007/s00222-011-0319-6}
}

@article{Zhang2017,
author = {Zhang, Jianlu and Zhang, Ke},
title = {Improved stability for analytic quasi-convex nearly integrable systems and optimal speed of Arnold diffusion},
journal = {Nonlinearity},
year = {2017},
month = {jun},
volume = {30},
number = {7},
pages = {2918},
doi = {10.1088/1361-6544/aa72b7},
url = {https://dx.doi.org/10.1088/1361-6544/aa72b7},
publisher = {IOP Publishing},
}

@ARTICLE{GuChieBen2014,
       author = {{Guzzo}, Massimiliano and {Chierchia}, Luigi and {Benettin}, Giancarlo},
        title = "{The steep Nekhoroshev's Theorem and optimal stability exponents (an announcement)}",
      journal = {Rendiconti Lincei. Scienze Fisiche e Naturali},
         year = 2014,
        month = aug,
       volume = {25},
       number = {3},
        pages = {293-299},
          doi = {10.4171/rlm/679},
archivePrefix = {arXiv},
       eprint = {1403.6776},
 primaryClass = {math-ph},
       adsurl = {https://ui.adsabs.harvard.edu/abs/2014RLSFN..25..293G}
}

@ARTICLE{boue,
       author = {{Bou{\'e}}, Gwena{\"e}l},
        title = "{The two rigid body interaction using angular momentum theory formulae}",
      journal = {Celestial Mechanics and Dynamical Astronomy},
         year = 2017,
        month = jun,
       volume = {128},
       number = {2-3},
        pages = {261-273},
          doi = {10.1007/s10569-017-9751-2},
archivePrefix = {arXiv},
       eprint = {1612.02556},
 primaryClass = {astro-ph.EP},
       adsurl = {https://ui.adsabs.harvard.edu/abs/2017CeMDA.128..261B}
}

@ARTICLE{Maciejewski,
       author = {{Maciejewski}, Andrzej J.},
        title = "{Reduction, Relative Equilibria and Potential in the Two Rigid Bodies Problem}",
      journal = {Celestial Mechanics and Dynamical Astronomy},
         year = 1995,
        month = mar,
       volume = {63},
       number = {1},
        pages = {1-28},
          doi = {10.1007/BF00691912},
       adsurl = {https://ui.adsabs.harvard.edu/abs/1995CeMDA..63....1M}
}

@ARTICLE{Lei2024,
       author = {{Lei}, Hanlun},
        title = "{Spin{\textendash}orbit coupling of the primary body in a binary asteroid system}",
      journal = {Celestial Mechanics and Dynamical Astronomy},
         year = 2024,
        month = oct,
       volume = {136},
       number = {5},
          eid = {37},
        pages = {37},
          doi = {10.1007/s10569-024-10211-5},
archivePrefix = {arXiv},
       eprint = {2407.21274},
 primaryClass = {astro-ph.EP},
       adsurl = {https://ui.adsabs.harvard.edu/abs/2024CeMDA.136...37L}
}

@ARTICLE{PSV2024,
       author = {{Pinzari}, Gabriella and {Scoppola}, Benedetto and {Veglianti}, Matteo},
        title = "{Spin orbit resonance cascade via core shell model: application to Mercury and Ganymede}",
      journal = {Celestial Mechanics and Dynamical Astronomy},
         year = 2024,
        month = oct,
       volume = {136},
       number = {5},
          eid = {39},
        pages = {39},
          doi = {10.1007/s10569-024-10207-1},
archivePrefix = {arXiv},
       eprint = {2402.07650},
 primaryClass = {math-ph},
       adsurl = {https://ui.adsabs.harvard.edu/abs/2024CeMDA.136...39P}
}

@BOOK{Cassels,
author = {Cassels, J.W.S.},
title = {An introduction to Diophantine approximation},
year = {1957},
publisher = {Cambridge University Press}
}

\end{document}